\documentclass[11pt,a4paper]{article}
\usepackage{a4wide}
\usepackage{graphicx}
\usepackage{amssymb,amsmath,amsthm,mathrsfs,xfrac}
\usepackage{mathtools}
\usepackage{enumerate,color}
\usepackage[colorlinks=true,citecolor=blue]{hyperref}
\usepackage{array}
\usepackage{amsfonts}
\usepackage{psfrag}
\usepackage{comment}

\newcommand{\R}{\mathbb{R}}
\newcommand{\N}{\mathbb{N}}

\newcommand{\dd}{\mathrm{d}}
\newtheorem{Alg}{Algorithm}[section]

\newtheorem{theorem}{Theorem}[section]
\newtheorem{lemma}{Lemma}[section]

\newtheorem{remark}{Remark}[section]

\newtheorem*{maintheorem*}{Main Theorem}

\title{Well-posedness and numerical analysis of an elapsed time model with strongly coupled neural networks}

\author{
\textsc{Mauricio Sep\'ulveda\thanks{\textbf{Corresponding author.} CI$^2$MA and DIM, Universidad de Concepci\'on, Concepci\'on-Chile. E-mail: \texttt{mauricio@ing-mat.udec.cl}}},
\textsc{Nicolas Torres \thanks{Departamento de Matemática Aplicada, Universidad de Granada-Espa\~na. E-mail: \texttt{torres@ugr.es}}} \, and      
\textsc{Luis Miguel Villada\thanks{GIMNAP-Departamento de Matem\'atica, Universidad del B\'io-B\'io and CI$^2$MA-Universidad de Concepci\'on, Concepci\'on-Chile. E-mail: \texttt{lvillada@ubiobio.cl}} }
}

\begin{document}
\maketitle
\begin{abstract}
The elapsed time equation is an age-structured model that describes the dynamics of interconnected spiking neurons through the elapsed time since the last discharge, leading to many interesting questions on the evolution of the system from a mathematical and biological point of view. In this work, we deal with the case when the transmission after a spike is instantaneous and the case with a distributed delay that depends on the previous history of the system, which is a more realistic assumption. Since the instantaneous transmission case is known to be ill-posed due to non-uniqueness or jump discontinuities, we establish a criterion for well-posedness to determine when the solution remains continuous in time, through an invertibility condition that improves the existence theory under more relaxed hypothesis on the nonlinearity, including the strongly excitatory case. Inspired in the existence theory, we adapt the classical explicit upwind scheme through a robust fixed-point approach and we prove that the approximation given by this scheme converges to the solution of the nonlinear problem through BV-estimates and we extend the idea to the case with distributed delay. We also show some numerical simulations to compare the behavior of the system in the case of instantaneous transmission with the case of distributed delay under different parameters, leading to solutions with different asymptotic profiles.
\end{abstract}

\noindent{\makebox[1in]\hrulefill}\newline


\section{Introduction}

The study of the precise mechanisms of brain processes has been a challenge for mathematicians and biologists with many interesting models including discrete systems, differential equations and stochastic processes. In particular, structured equations have been have proved to be a useful approach to understand how neurons interact and predict synchronous regular or irregular activities \cite{gerstner1995time,coombes2023neurodynamics,carrillo2025nonlinear}.

One of the seminal models in computational neuroscience developed by \cite{lapicque}, considered a neuron as part of an electrical network. This idea eventually was extended to models involving a large number of neurons, such as the integrate-and-fire systems, where the ensemble is described through the membrane potential by taking into account the internal and external voltage differences, leading to the well-known Fokker-Planck equations through a mean-field limit approach. These systems have been studied by many authors \cite{caceres2011analysis,carrillo2015qualitative,caceres2021understanding,carrillo2023noise,caceres2019global,ikeda2022theoretical,caceres2025asymptotic,caceres2024sequence} through different analysis techniques and numerical simulations.

In addition to considering the membrane potential to describe the behavior of neurons, age-structured models are an alternative approach to model the dynamics of interconnected spiking neurons. One of these equations is the well-known elapsed time model, where a given population of neurons is described by the {time elapsed} since their last discharge. In this network neurons are subjected to random discharges, which is related to changes in the membrane potential, stimulating other neurons to spike. One of the main motivations for this model is to predict brain activity through the previous history of spikes, and the key element determining the evolution of the system is the way neurons interact within the network, leading to different possible behaviors of neural activity and pattern formation.

Like the Fokker-Planck equation, this equation is a mean-field limit of a microscopic model that establishes a bridge of the dynamics of a single neuron with a population-based approach, whose aspects have been investigated in \cite{pham1998activity,ly2009spike,chevallier2015microscopic,chevallier2015mean,quininao2016microscopic,schwalger2019mind} and the connection between the two models has been studied in \cite{dumont2016noisy,dumont2020theoretical}.

This elapsed-time model was initially investigated in the pioneering works of
\cite{pakdaman2009dynamics,pakdaman2013relaxation,pakdaman2014adaptation} and then studied by many authors. Some important results on exponential convergence to the equilibrium for weak nonlinearities were proved in \cite{canizo2019asymptotic,mischler2018,mischler2018weak,caceres2025comparison} through different techniques such as the entropy method, semi-group theory, spectral arguments and comparison methods. Results on strong nonlinearities have been studied in \cite{torres2021elapsed}, where the existence of periodic solutions with jump discontinuities, while \cite{perthame2025strongly} proved convergence in the inhibitory case and results and the local stability was studied in \cite{caceres2025local}. Moreover, different extensions of the elapsed time model have been studied by incorporating new elements such as the fragmentation equation \cite{pakdaman2014adaptation}, spatial dependence with connectivity kernel in \cite{torres2020dynamics,dumont2024oscillations}, a multiple-renewal equation in \cite{torres2022multiple} and a leaky memory variable in \cite{fonte2022long}.

The classical elapsed time model with instantaneous transmission, which we will call throughout this article as \textit{Instantaneous Transmission Model} (ITM), is given by the following nonlinear age-structured equation (\cite{pakdaman2009dynamics})
\begin{equation}\label{eq:model}
\tag{ITM}
\begin{cases}
    \partial_tn+\partial_s n+p(s,N(t))n = 0,  & t>0,\: s>0, \\
    N(t)=n(t,s=0)= \int_0^{+\infty}p(s,N(t))n(t,s) \dd s,& t>0,\\
    n(0,s)=n^0(s)\geq0,& s\geq0,\\
    \int_0^{\infty}n^0(s)\dd s=1,
\end{cases}
\end{equation}
where $n(t,s)$ is the probability density of finding a neuron at time $t$, whose elapsed time since last discharge is $s\ge0$ and the function $N(t)$ represents the flux of discharging neurons.

For this equation, we assume that when a neuron spikes, its interactions with other neurons are instantaneous, so that we assume for simplicity in this model that the total activity of the network is simply given by the value of $N(t)$. The crucial nonlinearity is given by the function $p\colon[0,\infty)\times[0,\infty)\to[0,\infty)$, which is called the firing coefficient and it describes the susceptibility of neurons to spike. We assume that $p$ depends on the elapsed time $s$ and the activity $N$ and without loss of generality we consider that $p\in W^{1,\infty}(\R^+\times\R^+)$, though the regularity is not a crucial assumption as we will see in the numerical examples. We stick to these names of each term of the system, though they may differ in the literature. 

In this setting, neurons discharge at the rate given by $p$ and then the elapsed time is immediately reset to zero, as stated by the integral boundary condition of $n$ at $s = 0$. Following the terminology of age-structured equations, the elapsed time corresponds to the "age" of neuron. When a neuron discharges, it is considered to "die" where $p(s,N)n$ is the corresponding "death" term. After a neuron spikes, it instantaneously re-enter the cycle and it is considered to be "reborn" with the boundary condition at $s = 0$ representing the "birth" term.

Moreover, we assume that the coefficient $p$ is increasing with respect to the age $s$, which means that neurons are more prone to spike when the elapsed time since last discharge is large. According to the dependence of the coefficient $p$ on the total activity different regimes are possible. When $p$ is increasing with respect to $N$ we say that the network is excitatory, which means that under a high activity neurons are more susceptible to discharge. Similarly, when $p$ is decreasing we say that the network is inhibitory and we have the opposite effect on the network. Moreover, if the following conditions holds:
\begin{equation}
    \label{weakinter}
    \|\partial_N p\|_\infty<1,
\end{equation}
we say that the network is under a weakly interconnected regime, which means that the nonlinearity is weak. 

For the initial data $n^0\in L^1(\R^+)$ we assume that 
$n(t,s)$
is a probability density and we formally have the following mass-conservation property
\begin{equation}
\label{massconservation}
    \int_0^\infty n(t,s)\,\dd s = \int_0^\infty n^0(s)\,\dd s,\qquad\forall t\ge 0,
\end{equation}
which will be crucial in the analysis of \eqref{eq:model} equation. Throughout this article we consider solutions in the weak sense but for simplicity we simply refer to them as solutions.

Concerning stationary solutions of the nonlinear System \eqref{eq:model}, these are given by the solutions of the problem

\begin{equation}
\label{eqest1}
\begin{cases}
\partial_s n+p(s,N)n=0&s>0,\vspace{0.15cm}\\
N=n(s=0)\coloneqq\int_0^\infty p(s,N)n(s)\,\dd s,&\vspace{0.15cm}\\
\int_0^\infty n(s)\,\dd s=1,\quad n(s)\ge0.   
\end{cases}
\end{equation}

If the activity $N$ is given, we can determine the stationary density through the formula
\begin{equation}
\label{solest1}
n(s)=Ne^{-\int_0^sp(u,N)\,\dd u}.
\end{equation}
Thus by integrating with respect to $s$, we get that $(n,N)$ corresponds to a stationary solution of System \eqref{eq:model} if the activity satisfies the fixed-point equation

\begin{equation}
	\label{fixN1}
	N=F(N)\coloneqq\left(\int_0^\infty e^{-\int_0^sp(u,N)\,\dd u}\,\dd s\right)^{-1}.
\end{equation}

So that depending on the coefficient $p$, we have a unique solution in the inhibitory case and in the weak interconnections regime. For the excitatory case we may have multiple steady-states.

We also remark that a prototypical form of the function $p$ is given by
\begin{equation}
\label{p-esp}
p(s,N)=\varphi(N)\chi_{\{s>\sigma\}},
\end{equation}
which represents a firing coefficient with an absolute refractory period $\sigma>0$ so that for an age $s<\sigma$ neurons are not susceptible to discharge. For $s>\sigma$ neurons are able to discharge and the density $n$ decays exponentially according to $\varphi(N)$. The function $\varphi$ is assumed to be smooth and satisfies the following bounds
\begin{equation*}
p_0\le\varphi(N)\le p_1, \qquad\forall N\ge0,
\end{equation*}
for some constants $p_0,p_1>0$. This special case has been studied in \cite{pakdaman2009dynamics,torres2021elapsed}, where they proved convergence for the inhibitory case and constructed periodic solutions for the excitatory case. Another possible example is to consider a variable refractory period depending on the total activity
\begin{equation}
\label{p-esp2}
    p(s,N)=\chi_{\{s>\sigma(N)\}}.
\end{equation}
This type of firing coefficient was studied in \cite{pakdaman2013relaxation}, where existence of periodic solutions was also investigated.

One of the delicate issues of the case with instantaneous transmission is well-posedness. Indeed, if we look for a solution $n\in\mathcal{C}([0,T],L^1(\R^+))$ for some $T>0$, we observe that $N(0)$ satisfies the following equation

\begin{equation}
\label{eqN0}
    N(0)=\int_0^\infty p(s,N(0))n^0(s)\,\dd s.
\end{equation}

In the inhibitory case and the weak interconnections regime this equation has a unique solution for $N(0)$, while for the excitatory case we may have multiple solutions and thus the solution of \eqref{eq:model} equation is not unique. In addition we remark that in general $N(0)\neq n^0(0)$, which might imply that $n$ has discontinuities along the line $\{(t,s)\in\R^2\colon t=s\}$. 

Furthermore, when the system evolves in time we see that $N(t)$ is a solution of the fixed-point equation for a given $n\in\mathcal{C}([0,T],L^1(\R^+))$

\begin{equation}
\label{fixedNt}
    N(t)=\int_0^\infty p(s,N(t))n(t,s)\,\dd s,
\end{equation}

and depending on the evolution of $n(t,s)$, it is not clear how we can produce a continuous solution of system \eqref{eq:model}. Indeed, the multiplicity of solutions for the same initial data $n^0\in L^1(\R^+)$ and jump discontinuities has been observed in \cite{torres2021elapsed}.
This motivates to revisit well-posedness of the the instantaneous transmission model by studying how the solution of Equation \eqref{fixedNt} changes over time. 

With this idea, we also aim to make a numerical analysis of the nonlinear elapsed-time model by adapting the classical first-order explicit upwind scheme combined with the fixed-point approach of Equation \eqref{fixedNt} in order to deal with the potential ill-posedness of the case with instantaneous transmission. In particular, we will rely on an estimate of the discrete total variation to prove the convergence of the adapted scheme.

Previous studies on the numerical analysis of age-structured equations include those by 
\cite{lopez1991upwind,sulsky1994numerical,abia2005age,iannelli2017basic} and further generalizations to solutions in the space of positive regular measures $\mathcal{M}^+(\R^+)$ have been investigated by \cite{carrillo2014splitting,gwiazda2010nonlinear,ackleh2019finite} through the particle method. In particular, the instantaneous transmission equation can be regarded as a singular case of the general autonomous nonlinear age-structured model
\begin{equation}
  \begin{cases}
\partial_t n+\partial_s n+\mu(s,\Phi(t))n=0,&t>0,\,s>0,\vspace{0.15cm}\\
N(t)\coloneqq n(t,0)=\int_0^\infty \beta(s,\Phi(t))n(t,s)\,\dd s,&t>0,\vspace{0.15cm}\\
n(0,s)=n^0(s),\:\Phi(t)=\int_0^\infty\gamma(s)n(t,s)\,ds, & s\ge 0,   
\end{cases}  
\end{equation}
where $\beta,\mu\in L_+^\infty(\R^+\times\R^+)$ are birth and death coefficients respectively, while the $\gamma\in L_+^\infty(\R^+)$ is the niche coefficient with $\Phi(t)$ as the total occupation of the niche. In this context, \eqref{eq:model} equation formally corresponds to $\mu\equiv\beta$ and $\gamma(s)=\delta_0(s)$, a Dirac mass. This singularity motivates us to adapt the classical upwind scheme in order to deal with the fixed-point equation \eqref{fixedNt}.

On the other hand, one way to avoid these issues on the well-posedness of \eqref{eq:model} equation is to consider from a biophysical point of view that when neurons spike there exists a delay in the transmission to other neurons. In order to take into account this effect, 
\cite{pakdaman2009dynamics} considered a modification of the elapsed time model incorporating a distributed delay, which corresponds to the following variant of \eqref{eq:model} equation that we will call as \textit{Distributed Delay Model} (DDM)
\begin{equation}\label{eq:delay}
\tag{DDM}
\begin{cases}
    \partial_tn+\partial_s n+p(s,X(t))n = 0,  & t>0,\: s>0, \\
    N(t)=n(t,s=0)= \int_0^{+\infty}p(s,X(t))n(t,s) \dd s,& t>0,\\
    X(t)=\int_0^t \alpha(t-\tau)N(\tau)\,\dd \tau, & t>0,\\
    n(0,s)=n^0(s)\geq0,& s\geq 0,\\
    \int_0^{\infty}n^0(s)\dd s=1. &
\end{cases}
\end{equation}
The kernel $\alpha\in L^1(\R^+)$ with $\alpha\ge0$, corresponds to the distributed delay and for simplicity we may assume that $\alpha$ is smooth and uniformly bounded with $\int_0^\infty\alpha(\tau)\dd \tau=1$, but the theoretical results are still valid if we only assume that $\alpha$ is integrable. For this model $N(t)$ is the flux of discharging neurons and $X(t)$ is the total activity, which depends on the values taken by $N$ in the past (i.e. in the interval $[0,t]$) through the convolution with $\alpha$. Unlike the instantaneous transmission model, the coefficient $p$ depends on the total activity $X(t)$ instead of the discharging flux. Properties such as mass conservation remain valid for this modified model. We remark that under the condition $\int_0^\infty\alpha(\tau)\dd \tau=1$, we consider the steady states of equation \eqref{eq:delay} to be the same as those of \eqref{eq:model}.

Like the case of Equation \eqref{fixedNt} in the \eqref{eq:model} system, we obtain the analogous fixed-point problem for $X(t)$ for a given density $n\in\mathcal{C}([0,T],L^1(\R^+))$
\begin{equation}
 X(t)=\int_0^t\int_0^\infty\alpha(t-\tau)p(s,X(\tau))n(\tau,s)\,\dd s\,\dd \tau,    
\end{equation}
so that we can also extend the fixed-point approach to the \eqref{eq:delay} equation in order to extend the numerical analysis of the \eqref{eq:model} equation.

When the kernel $\alpha(t)$ is smooth, the distributed delay model equation corresponds to a regularized version in time of \eqref{eq:model} equation, through an approximation of the $N(t)$. In particular when $\alpha(t)$ approaches in the sense of distributions to the Dirac's mass $\delta(t)$, then we formally get $X(t)=N(t)$ and thus we recover the case with instantaneous transmission \eqref{eq:model}. In this context, an interesting question is to determine how good a simplification it is to assume instantaneous transmission in the dynamics of interconnected spiking neurons.

Another important example studied in \cite{pakdaman2009dynamics} is the exponential delay given by $\alpha(t)=\frac{1}{\lambda}e^{-t/\lambda}$ so that $X(t)$ satisfies the following differential equation

\begin{equation}
\label{edo-X}
    \begin{cases}
        \lambda X'(t) +X(t)=N(t),\\
        X(0)=0.
    \end{cases}
\end{equation}
giving a simple way to compute numerical solution of this system.

Similarly, when $\alpha(t)$ approaches in the sense of distributions to the Dirac's mass $\delta(t-d)$, then we formally get $X(t)=N(t-d)$ and we recover a version of the classical elapsed time equation with a single discrete delay $d$

\begin{equation}\label{eq:delaypuro}
\begin{cases}
    \partial_tn+\partial_s n+p(s,N(t-d))n = 0,  & t>0,\: s>0, \\
    N(t)=n(t,s=0)= \int_0^{+\infty}p(s,N(t-d))n(t,s) \dd s,& t>0,\\
    n(0,s)=n^0(s)\geq0,& s\geq 0,\\
    \int_0^{\infty}n^0(s)\dd s=1.
\end{cases}
\end{equation}

Given the flexibility of the the distributed delay model through the kernel $\alpha$, in this article we will focus on the \eqref{eq:model} and \eqref{eq:delay} equations, but the results and techniques presented in this work 
are also valid for Equation \eqref{eq:delaypuro}.

\begin{remark}
\label{past}
    The most general form of the equation \eqref{eq:delay} is when the total activity satisfies the following equation:
    \begin{equation*}
    X(t)=\int_{-\infty}^t\alpha(t-s)N(s)\,ds,
    \end{equation*}
    where $N(t)=N^0(t)$ for a given function $N^0$ defined when $t<0$, that acts as an extra initial condition for the distributed delay problem. As it was considered in \cite{pakdaman2009dynamics}, we assume for simplicity that there is no activity in the past, i.e. $X(t)=N(t)=0$ for $t<0$. The results on existence and the corresponding numerical analysis are easily adapted for a given past $N^0$ with compact support in $(-\infty,0]$.
\end{remark}

This article is devoted to revisiting well-posedness in order to adapt the classical upwind scheme by incorporating fixed-point arguments and giving a unified analysis of well-posedness for both \eqref{eq:model} and \eqref{eq:delay} equations, improving the proofs given in \cite{pakdaman2009dynamics,canizo2019asymptotic}. In the work of 
\cite{pakdaman2009dynamics} they proved well-posedness when the coefficient $p$ is of the form given by \eqref{p-esp} or \eqref{p-esp2} under some asymptotic conditions in the growth of $p$ with respect to the discharging flux $N$ (or the total activity $X$ in the case of distributed delay), while in the work of Cañizo et al. \cite{canizo2019asymptotic} they mainly focused in the weak interconnections regime for \eqref{eq:model} equation.

The article is organized as follows. In Section \ref{ITM}, we study the instantaneous transmission model by revisiting well-posedness in the inhibitory case (including the weak interconnections case as well) and explaining how the arguments can be extended for the general excitatory case. Then we proceed to explain the scheme to solve numerically equation \ref{eq:model} and prove the necessary BV estimates that ensure the convergence of the numerical method. In Section \ref{DDM}, we make an analogous analysis for the case with distributed delay by adapting the ideas and arguments applied to the the case with instantaneous transmission. Finally, in Section \ref{Simu}, we present numerical simulations to compare both models under different parameters choices, including the inhibitory and excitatory regime. In particular, we consider different types of delay kernel $\alpha$ in order to observe the limit cases when $\alpha(t)$ is approaching a Dirac mass and the possible asymptotic behavior thereof. This extends the numerical simulations made by 
\cite{pakdaman2009dynamics}, where they consider mainly the exponential kernel in the case with distributed delay.

\section{Instantaneous Transmission Model (ITM)}
\label{ITM}

\subsection{Well-posedness of ITM revisited}
In this subsection we prove that the solution of \eqref{eq:model} equation is well-posed in the inhibitory case and the weak interconnections regime. We improve the ideas of 
\cite{pakdaman2009dynamics} by considering more general forms for the coefficient $p$ and we improve the result of the weak interconnection regime of \cite{canizo2019asymptotic} by extending the existence and uniqueness of a solution when we drop the absolute value in condition \eqref{weakinter}. The main idea of the proof is to propose the appropriate fixed-point problem that eventually leads to a solution of \eqref{eq:model} through the contraction principle.

\begin{theorem}
\label{wellp-nodelay}
     Consider a nonnegative $n^0\in L^1(\R^+)$. Assume that $p\in W^{1,\infty}(\R^+\times\R^+)$ and let $\gamma \coloneqq \sup_{s,N}\partial_N p(s,N)$ with $\gamma\|n^0\|_1<1$, then \eqref{eq:model} equation has a unique solution $n\in \mathcal{C}\left([0,\infty),L^1(\R^+)\right)$ and $N\in\mathcal{C}[0,\infty)$. Moreover $n$ satisfies the property of mass conservation \eqref{massconservation}.
\end{theorem}

One key remark on the well-posedness of \eqref{eq:model} equation in the general case is the following. 

\begin{remark}
    We note that the regularity of the coefficient $p$ is not fundamental for the proof and Theorem \ref{wellp-nodelay} is still valid for a wider class of functions such as 
  \begin{equation*}
  p(s,N)=\varphi(N)\chi_{\{s>\sigma(N)\}},      
  \end{equation*}
with Lipschitz bounded functions $\varphi$ and $\sigma$, as they are studied in \cite{pakdaman2009dynamics,torres2021elapsed}. So under a similar conditions for the inhibitory and weakly excitatory regimes, we can apply the arguments used in Theorem \ref{wellp-nodelay} to get well-posedness in these cases.

    Moreover, the proof of the Theorem \ref{wellp-nodelay} can be replicated in the general excitatory case to prove the existence of solutions with $n\in\mathcal{C}\left([0,T],L^1(\R^+)\right)$ for some $T>0$. Indeed, from the proof Lemma \ref{fixedpoint-itm} we can apply the implicit function theorem as long as the following invertibility condition holds
    \begin{equation}\label{cond:inverti}
     \Psi(N(t),n(t,\cdot)):= 1-\int_0^\infty \partial_N p(s,N(t))n(t,s)\dd s  \neq 0,\qquad\forall t\in [0,T],  
    \end{equation}
    where $\Psi:\R^+\times L^1(\R^+)\longmapsto\R$,
    so that we obtain existence of a solution (or possible branches of solutions) of \eqref{eq:model} equation defined locally in time by applying the arguments in the proof of Theorem \eqref{wellp-nodelay}. If $\Psi(N(t^*),n(t^*,\cdot))=0$ for some $t^*>0$, then the continuity of solutions is not ensured and jump discontinuities might arise. We explore this aspect in the section of numerical simulations.
\end{remark}

For the proof we need the following lemmas, which will be the key idea throughout this article. We start with following result on the linear case.

\begin{lemma}
Assume that $n^0\in L^1(\R^+)$ and $p\in L^\infty(\R^+\times\R^+)$. Then for a given $N\in\mathcal{C}[0,T]$, the linear equation
\begin{equation}
\label{linearaux}
\begin{cases}
    \partial_tn+\partial_s n+p(s,N(t))n = 0,  & t>0,\: s>0, \\
    n(t,s=0)=\int_0^{+\infty}p(s,N(t))n(t,s) \dd s,& t>0,\\
    n(0,s)=n^0(s)\geq0,& s\geq 0,
    \end{cases}    
\end{equation}
has a unique weak solution $n\in\mathcal{C}([0,T],L^1(\R^+))$. Moreover $n$ is nonnegative and verifies the mass conservation property \eqref{massconservation} for $t\in[0,T]$.
\end{lemma}
\begin{proof}
    For a proof see the linear theory developed in \cite{perthame2006transport,pakdaman2009dynamics,canizo2019asymptotic}.
\end{proof}

By using the implicit function theorem, we prove the following lemma that is the keystone in the proof of our theorem.

\begin{lemma}
\label{fixedpoint-itm}
    Consider a nonnegative function $n^0\in L^1(\R^+)$. Let $\gamma \coloneqq \sup_{s,N}\partial_N p(s,N)$ and assume that $\gamma\|n^0\|_1<1$. Then there exists a unique solution for $N$ of the equation
    \begin{equation}
    \label{eqfix-N}
        N=F(N)\coloneqq \int_0^\infty p(s,N)n^0(s)\,\dd s,
    \end{equation}
    that we call $N\coloneqq\psi(n^0)$, where the map $\psi\colon L^1(\R^+)\to\R$ satisfies the following estimate
    \begin{equation}
    \label{psi-lip-l1}
        |\psi(n^1)-\psi(n^2)|\le \frac{\|p\|_\infty}{1-\gamma\|n^0\|_1}\int_0^\infty|n^1-n^2|(s)\,\dd s,
    \end{equation}
    for nonnegative integrable functions $n^1,n^2$ with $\|n^1\|_1=\|n^2\|_1=\|n^0\|_1$.
\end{lemma}

\begin{proof}
    Observe that $F$ is a continuous and bounded with respect to $N$. Indeed, for all $N$ we have 
    \begin{equation*}0\le F(N)\le \|p\|_\infty\|n^0\|_1,\end{equation*}
    and hence there exists $N\in [0,\|p\|_\infty\|n^0\|_1]$ such that $N=F(N)$. Moreover the function $g(N)=N-F(N)$ is strictly increasing. Indeed,
    \begin{equation*}g'(N)=1-F'(N)=1-\int_0^\infty\partial_N p(s,N)n^0(s)\,\dd s\ge 1-\gamma\|n^0\|_1>0,\end{equation*}
    and therefore the $N=F(N)$ has a unique solution that we call $\overline{N}=\psi(\overline{n})$. Consider the set $U$ defined by
    \begin{equation*}U\coloneqq\{n\in L^1(\R^+)\colon \gamma\|n\|_1<1\},\end{equation*}
    and  the map $G\colon X\times\R\to\R$ defined by
    \begin{equation*}G(n,N)=N-\int_0^\infty p(s,N)n^0(s)\,\dd s.\end{equation*}
    Observe that for $(n,N)\in U\times\R$ with $G(n,N)=0$ we get $\partial_N G(n,N)>0$. By the implicit function theorem, we notice that $\psi$ it is a differentiable map on $U$. Moreover $D\psi\colon L^1(\R)\to\R$ at a point $(n,N)\in U\times[0,\infty)$  is given by\begin{equation*}D\psi[h]=\frac{\int_0^\infty p(s,N)h(s)\,\dd s}{1-\int_0^\infty\partial_N p(s,N)n^0(s)\,\dd s},\end{equation*}
    and we have the following estimate in the operator norm at the point $(n^0,N^0)$
\begin{equation*}
    \|D\psi\|\le \frac{\|p\|_\infty}{1-\gamma\|n^0\|_1},
\end{equation*}
    thus for $n^1,n^2$ with the same norm $n^0$, the inequality \eqref{psi-lip-l1} readily follows.
\end{proof}

With this lemma, we continue the proof of the Theorem \ref{wellp-nodelay}.

\begin{proof}
    Consider $T>0$ and a fixed nonnegative $N\in\mathcal{C}[0,T]$. For this function $N$ we define $n[N]\in\mathcal{C}([0,T],L^1(\R^+)$ as the weak solution of the linear equation \eqref{linearaux},
    which can be expressed trough the method of characteristics
    \begin{equation*}
    n(t,s)=n^0(s-t)e^{-\int_0^t p(s-t+t',N(t'))\dd t'}\chi_{\{s>t\}}+N(t-s)e^{-\int_0^s p(s',N(t-s+s'))\dd s'}\chi_{\{t>s\}}.
    \end{equation*}
    From the mass conservation property we know that $\|n[N](t)\|_1=\|n^0\|_1$.
    
    On the other hand, for a given $n\in\mathcal{C}([0,T],L^1(\R^+))$, let $\psi\colon \mathcal{C}([0,T],L^1(\R^+))\to\mathcal{C}[0,T]$ the unique solution $X\in\mathcal{C}[0,T]$ of the equation
       \begin{equation*}
    X=\int_0^\infty p(s,X)n(t,s)\,\dd s.
   \end{equation*}
    Under this setting we get a solution of the nonlinear \eqref{eq:model} equation if only if we find $N\in\mathcal{C}[0,T]$ such that is a solution of the equation
      \begin{equation*}
N=H(N)\coloneqq\psi\left(n[N]\right),
\end{equation*}
    where $H\colon\mathcal{C}[0,T]\to\mathcal{C}[0,T]$.
    
    We assert that for $T$ small enough the map $H$ is a contraction. For nonnegative functions $N_1,N_2\in\mathcal{C}[0,T]$, we get from the method of characteristics we get the following inequality
  \begin{equation*}
    \int_0^\infty |n[N_1]-n[N_2]|(t,s)\,\dd s\le A_1+A_2+A_3,
  \end{equation*}
    where $A_1,A_2,A_3$ are given by
    \begin{eqnarray*}
        A_1&=&\int_0^\infty n^0(s)\left|e^{-\int_0^t p(s+t',N_1(t'))\dd t'}-e^{-\int_0^t p(s+t',N_2(t'))\dd t'}\right|\,\dd s,\\
        A_2&=&\int_0^t |N_1-N_2|(t-s) e^{-\int_0^s p(s',N_1(t-s+s'))\dd s'}\dd s,\\
        A_3&=&\int_0^t N_2(t-s)\left|e^{-\int_0^s p(s',N_1(t-s+s'))\dd s'}-e^{-\int_0^s p(s',N_2(t-s+s'))\dd s'} \right|\dd s.
    \end{eqnarray*}
    We proceed by estimating each term. For simplicity we write $n_1=n[N_1],n_2=n[N_2]$, so that for $A_1$ we have
    $$A_1\le \int_0^\infty n^0(s)\int_0^t|p(s+t',N_1(t'))-p(s+t',N_2(t'))|\,\dd t'\,\dd s \le T\|n^0\|_1 \|\partial_N p\|_\infty \|N_1-N_2\|_\infty,$$
    while for $A_2$ we have
 \begin{equation*}
    A_2\le \int_0^t |N_1-N_2|(t-s) \dd s \le T\|N_1-N_2\|_\infty,
 \end{equation*}
    and for $A_3$ we get
    \begin{equation*}
    \begin{split}
        A_3 &\le \|p\|_\infty\|n^0\|_1 \int_0^t\int_0^s|p(s',N_1(t-s+s'))-p(s',N_2(t-s+s'))|\dd s'\,\dd s\\
        &\le\frac{T^2}{2} \|p\|_\infty\|n^0\|_1\|\partial_N p\|_\infty\|N_1-N_2\|_\infty. 
    \end{split}
    \end{equation*}
    By combining these inequalities we have finally
    \begin{equation}
    \label{nN1-nN2}
       \sup_{t\in [0,T]}\|n_1(t,\cdot)-n_2(t,\cdot)\|_1\le \left(\|n^0\|_1\|\partial_N p\|_\infty+\frac{T}{2}\|p\|_\infty\|n^0\|_1\|\partial_N p\|_\infty+1\right)T\|N_1-N_2\|_\infty.
    \end{equation}

    Since the mass is conserved, we apply Lemma \ref{fixedpoint-itm} to get
    \begin{equation}
    \label{psi1-psi2}
        \sup_{t\in [0,T]}|\psi(n^1(t,\cdot))-\psi(n^2(t,\cdot))|\le \frac{\|p\|_\infty}{1-\gamma\|n^0\|_1}\sup_{t\in[0,T]}\|n^1(t,\cdot)-n^2(t,\cdot)\|_1,  
    \end{equation}
    and therefore, we deduce from \eqref{nN1-nN2} and \eqref{psi1-psi2} that the following estimate holds for $H$
\begin{equation*}
        \|H(N_1)-H(N_2)\|_\infty\le \frac{\|p\|_\infty T}{1-\gamma\|n^0\|_1}\left(\frac{T}{2}\|p\|_\infty\|n^0\|_1\|\partial_N p\|_\infty+\|n^0\|_1\|\partial_N p\|_\infty+1\right)\|N_1-N_2\|_\infty,
\end{equation*}
    so that for $T>0$ small enough we get a unique fixed point of $H$ by contraction principle, implying the existence of a unique solution $n\in \mathcal{C}\left([0,T],L^1(\R^+)\right)$ with $N\in\mathcal{C}[0,T]$ of \eqref{eq:model} equation. Since the mass is conserved, we can iterate this argument to conclude the solution is indeed defined for all $T>0$.
\end{proof}

\subsection{Numerical scheme for ITM}
In this subsection, a first-order explicit upwind scheme for the \eqref{eq:model} equation is introduced based on the finite-volume framework \cite{leveque1992numerical,godlewski1996numerical,bouchut2004nonlinear}. We consider an uniform discretization of $\Omega=[0,T]\times[0,+\infty)$   with cells $I_j=[s_{j-\frac12},s_{j+\frac12})$,  interface points $s_{j+\frac12}=j\Delta s>0$ and centers $s_j=(j-\frac12)\Delta s$, $j\in\N$ such that $[0,+\infty)=\cup_{j\in\N}I_j$ and $I^m=[t^{m-1},t^{m}]$, $t^{m}=m\Delta t>0$, $m\in\N$, $T=M\Delta t$ such that $[0,T]=\cup_{m=0}^M I^m$. Let $n_j^m=\frac{1}{\Delta s}\int_{I_j}n(t^m,s)\dd s$ be the cell average of $n(t,s)$ at time $t^m$ in the cell $I_j$, then applying an explicit finite-volume approximation with an upwind discretization for the convective term in \eqref{eq:model} equation, we obtain 
\begin{equation}
\label{finit-diff}
    \frac{n_j^{m+1}-n_j^m}{\Delta t}+\frac{n_j^m-n_{j-1}^m}{\Delta s}+p(s_j,N(t^{m}))n_j^m=0,\quad j\in\N,\, m\in\N,
\end{equation}
and  
\begin{equation*}
    N(t^m)=n(t^m,0)\approx \Delta s \sum_{j\in\N} p(s_j,N(t^m))n_j^m,\quad m\in\N.
\end{equation*}
Now, if we define $N^m:=N(t^m)$, the solution of the partial differential equation in \eqref{eq:model} can be solved by the explicit upwind scheme
\begin{equation}\label{metexp}
    n_j^{m+1}=n_j^m-\frac{\Delta t}{\Delta s}(n_{j}^m-n_{j-1}^m)-\Delta t p(s_j,N^m)n_j^m,\quad j\in\N,\,m\in\N.
\end{equation}
In particular for $j=1$ we have
\begin{equation*} 
    n_1^{m+1}=n_1^m-\frac{\Delta t}{\Delta s}(n_{1}^m-N^m)-\Delta t p(s_1,N^m)n_1^m.
\end{equation*}
The explicit upwind scheme \eqref{metexp} is stable if the CFL condition hold
\begin{equation*}
    1-\Delta t\left( \frac{1}{\Delta s}+p(s_j,N^m) \right)\geq 0,\quad j\in\N,\,m\in\N.
\end{equation*}

In that regard, the numerical scheme for $n_j^m$ and $N^m$ can be summarized in the following algorithm.

\begin{Alg}{ITM numerical scheme }\label{Alg}
 \begin{itemize}
 \item[] {Input: Approximate initial data $\{n^0_j\}_{j\in\N}$}
 \item[] Solve for $N^0$
    \begin{equation}
    \label{eqfijo-N0}
     N^{0}= \sum_{j\in\N}\Delta s p(s_j,N^{0}) n_j^{0}.
    \end{equation}  
 \item[] Choose $\Delta t$ such that 
 \begin{equation}\label{CFL}
     \Delta t \le \left(\frac{1}{\Delta s}+\|p\|_\infty\right)^{-1}.
 \end{equation}
 \item[] {\bf For} $m\in\N_0$ {\bf do}
 \item[] \hspace{0.5cm} {\bf For}  $j\in\N$ {\bf do}
 \begin{equation*}
     n_j^{m+1} \leftarrow
     \begin{cases}
      n_1^m-\frac{\Delta t}{\Delta s}(n_{1}^m-N^m)-\Delta t p(s_1,N^m)n_1^m & j=1,\\
      n_j^m-\frac{\Delta t}{\Delta s}(n_{j}^m-n_{j-1}^m)-\Delta t p(s_j,N^m)n_j^m & j>1.
     \end{cases}
 \end{equation*}
 \item[] \hspace{0.5cm} {\bf end} 
 \item[] \hspace{0.5cm} Solve for $N^{m+1}$ 
 \begin{equation}
 \label{eqfijo-N}
     N^{m+1}= \sum_{j\in\N}\Delta s p(s_j,N^{m+1}) n_j^{m+1}.
 \end{equation}
 \item[] {\bf end}
 \item[] { Output: Approximate solution $\{n_j^{m+1}\}_{j\in\N}$ and $N^{m+1}$ at time $t^{m+1}=(m+1)\Delta t$}
 \end{itemize} 
\end{Alg}

The solution of Equation \eqref{eqfijo-N} for $N^{m+1}$ can be solved with different numerical methods such as Newton-Raphson, bisection or inverse quadratic interpolation. In particular for the weak interconnections regime, the solution of Equation \eqref{eqfijo-N} $N^{m+1}$ can be approximated in terms of $N^{m}$ through the following formula if $\Delta t$ is small enough:
$$N^{m+1}= \sum_{j\in\N}\Delta s p(s_j,N^{m}) n_j^{m+1}.$$
For simplicity in the estimates we assume that we can compute the solution of Equation \eqref{eqfijo-N} exactly, but the results remain valid if we take into account a specific method to get an approximation.

\begin{remark}
 Analogously to Equation \eqref{fixedNt}, we will prove in this section that Equations \eqref{eqfijo-N0} and \eqref{eqfijo-N} have a unique solution in the inhibitory case and the weak interconnections regime. In the excitatory case, we may have multiple solutions for $N^0$ that lead to different branches of numerical solutions for the ITM equation \eqref{eq:model} defined in some interval of time. Depending on how we calculate the solution of the fixed-point problem \eqref{eqfijo-N}, the numerical method will approximate one of the multiple possible solutions.   
\end{remark}

In order to prove the convergence of the upwind scheme in the Algorithm \ref{Alg}, we follow the ideas of the previous subsection on well-posedness and we prove a BV-estimate that will be the crucial in the analysis. For simplicity we assume that the initial data $n^0$ is compactly supported, but the theoretical results still hold when the initial data $n^0$ vanishes at infinity.

We start with some lemmas that will be useful in the sequel. 
 \begin{lemma}{(\textbf{$L^{1}$}-norm)}
 \label{L1-Mass}Numerical approximation obtained with the Algorithm \ref{Alg} satisfies 
 \begin{equation*}
     \|n^{m}\|_1:=\sum_{j\in\N}\Delta s n_j^m=\|n^{0}\|_1, \quad m\in\N.
 \end{equation*}
 \end{lemma}
 \begin{proof}
 Multiply equation \eqref{metexp} by $\Delta s$ and sum over $j\in\N$ and using the boundary condition we obtain 
\begin{equation*}
    \sum_{j\in\N}\Delta s n_j^{m+1}=\sum_{j\in\N}\Delta s n_j^m-\Delta t \left( N^m- \sum_{j\in\N}\Delta s p(s_j,N^m)n_j^m \right).
\end{equation*}
 \end{proof}

\begin{lemma}{(\textbf{$L^{\infty}$}-norm)}
\label{L-infty}
Assume the initial data $n^0\in L^\infty(\R^+)$ is nonnegative. Then under CFL restriction \eqref{CFL} the numerical solution obtained with the Algorithm \ref{Alg} satisfies
\begin{eqnarray*}
 0\le n^m_j\le \|n^0\|_\infty, & 0\le N^m\le \|p\|_\infty\|n^0\|_1, & \textrm{for all}\:j,m\in\N.
\end{eqnarray*}
\end{lemma}

\begin{proof}
By the CFL condition, observe that 
\begin{equation*}
0\le n_j^{1}=\left(1-\Delta t\left( \frac{1}{\Delta s}+p(s_j,N^0) \right) \right)n_j^0+\frac{\Delta t}{\Delta s}n_{j-1}^0\le \|n^0\|_\infty, \quad\textrm{for all}\:j\in\N,
\end{equation*}
and by induction in $m$ we conclude that $0\le n^m_j\le\|n^0\|_\infty$ for all $j,m\in\N$. Now for the estimates involving $N^m$, we apply the previous lemma to conclude the following inequality
\begin{equation*}
0\le N^m\le \|p\|_\infty\sum_{j\in\N}\Delta s n^m_j = \|p\|_\infty\sum_{j\in\N}\Delta s n^0_j= \|p\|_\infty\|n^0\|_1,
\end{equation*}
and we get the desired result.
\end{proof}

We now prove that for each iteration in $m$ of the numerical method. Equation \eqref{eqfijo-N} has a indeed a unique solution $N^m$ in the inhibitory case and the weak interconnection regime. The following result corresponds to the discrete version of Lemma \ref{fixedpoint-itm}.

\begin{lemma}
\label{fixedpoint-discret}
    Consider a discretization $n^0=(n^0_j)_{j\in\N}$ of nonnegative terms. Let $\gamma \coloneqq \sup_{s,N}\partial_N p(s,N)$ and assume that $\gamma\|n^0\|_1<1$. Then there exists a unique solution for $N$ of the equation
 \begin{equation*}
    N=F(N)\coloneqq\sum_{j\in\N}\Delta s p(s_j,N)n^0_j,
 \end{equation*}
    that we call $n^0\coloneqq\psi(n^0)$. Moreover the map $\psi$ satisfies the following estimate
    \begin{equation}
        \label{psi-lip}
        |\psi(n^1)-\psi(n^2)|\le \frac{\|p\|_\infty}{1-\gamma\|n^0\|_1}\sum_{j\in\N}\Delta s|n^1_j-n^2_j|,
    \end{equation}
    for sequences $n^1=(n^1_j)_{j\in\N},\,n^2=(n^2_j)_{j\in\N}$ of nonnegative terms with $\|n^1\|_1=\|n^2\|_1=\|n^0\|_1$.
\end{lemma}

\begin{proof}
    The proof is similar to that of Lemma \ref{fixedpoint-itm} by the replacing the integral terms with discrete summations.
\end{proof}

\begin{remark}
    In the excitatory case, the ideas of the previous lemma can be applied as long as the following invertibility condition holds
    \begin{equation}
      \Psi(N,n):= 1-\Delta s\sum_{j\in\N} \partial_N p(s_j,N)n_j\dd s\neq 0,  
    \end{equation}
    so that we can extend the numerical solution through the implicit function theorem in order to approximate a continuous solution of the equation \eqref{eq:delay} in some interval $[0,T]$. This is the analog to the invertibility condition \eqref{cond:inverti}.
\end{remark}

We now establish some BV-lemmas on the discretization given by the upwind scheme to prove that the numerical method approximates the solution of \eqref{eq:model} equation when $\Delta t$ and $\Delta s$ converge to zero. For the discretization $n=(n_j)_{j\in\N}$ we define the total variation as
\begin{equation}
  TV(n)\coloneqq\sum_{j=0}^{\infty}|n_{j+1}-n_j|.  
\end{equation}
In this context, we prove the following key lemma.

\begin{lemma}{(\textbf{BV}-estimate)}
\label{BV-lem}
Assume that $TV(n^0)<\infty$ and the CFL condition \eqref{CFL}. Then there exist constants $C_1,C_2>0$ (depending only $p$ and the norms of $n^0$) such that for $m\in\N$ we have
\begin{equation}
    TV(n^m)\le e^{C_1T} TV(n^0)+C_2(e^{C_1T}-1),
    \label{BV_estimate}
\end{equation}
with  $T=m\Delta t$ and $TV(n^m)=\displaystyle\sum_{j=0}^\infty|n_{j+1}^{m}-n_{j}^{m}|$.
\end{lemma}

\begin{proof}
Using the notation $\Delta^{+}n_j^m=n_{j+1}^{m}-n_{j}^{m}$, we have
\begin{eqnarray*}
   \Delta^{+}n_j^{m+1}&=& \Delta^{+}n_j^{m}-\frac{\Delta t}{\Delta s}\left(\Delta^{+}n_j^{m}-\Delta^{+}n_{j-1}^{m}\right)-\Delta t \left(p(s_{j+1},N^m)n_{j+1}^m-p(s_{j},N^m)n_j^m\right),\\
   &=& \Delta^{+}n_j^{m}-\frac{\Delta t}{\Delta s}\left(\Delta^{+}n_j^{m}-\Delta^{+}n_{j-1}^{m}\right)-\Delta t p(s_{j+1},N^m)\Delta^{+}n_j^{m} -\Delta t n_j^m(p(s_{j+1},N^m)-p(s_{j},N^m)).
\end{eqnarray*}
now applying  \eqref{p-esp} and taking absolute value and using the CFL condition \eqref{CFL} we obtain

\begin{equation*}
     |\Delta^{+}n_j^{m+1}| \le \left(1-\Delta t\left( \frac{1}{\Delta s}+p(s_{j+1},N^m) \right) \right)|\Delta^{+}n_j^{m}|
     +\frac{\Delta t}{\Delta s}|\Delta^{+}n_{j-1}^{m}|+\Delta t\Delta s\|\partial_s p\|_\infty n_j^m.
\end{equation*}

Now by summing over all $j\ge1$, we have 
\begin{equation*}
    \begin{split}
       \sum_{j=1}^\infty|\Delta^{+}n_j^{m+1}|&\le \sum_{j=1}^\infty|\Delta^{+}n_j^{m}|+\frac{\Delta t}{\Delta s}|\Delta^{+}n_{0}^{m}|
-\Delta t \sum_{j\in\N}^\infty p(s_{j+1},N^m) |\Delta^{+}n_j^{m}| +\Delta t\|\partial_s p\|_\infty\sum_{j=1}^{\infty}\Delta s n_j^m,
    \end{split}
\end{equation*}
and from the mass conservation we deduce
\begin{equation}
\label{cota2}
\sum_{j=1}^\infty|\Delta^{+}n_j^{m+1}|\le
\sum_{j=1}^\infty|\Delta^{+}n_j^{m}|
+ \frac{\Delta t}{\Delta s} |\Delta^{+}n_{0}^{m}|+\Delta t\|\partial_s p\|_\infty\|n^0\|_1. 
\end{equation}

On the other hand, by assuming $\Delta t$ small enough we have
\begin{equation}
\label{cota3}
\begin{split}
     |\Delta^{+}n_0^{m+1}| & = |n_1^{m+1}-N^{m+1}|\\
     & \le \left(1-\frac{\Delta t}{\Delta s}\right)|n_1^m-N^m|+  |N^{m+1}- N^m|+\Delta t p(s_1,N^m) n^m_1\\
     &\le \left(1-\frac{\Delta t}{\Delta s}\right)|\Delta^+ n_0^m|+  |N^{m+1}- N^m|+\Delta t \|p\|_\infty \|n^0\|_\infty.
\end{split} 
\end{equation}

Furthermore, from Lemma \ref{fixedpoint-discret} there exists $C_1>0$ depending only on $p$ such that
\begin{eqnarray}
 |N^{m+1}- N^m| 
 &=&
 \left| \psi(n^{m+1})-\psi(n^{m})\right|\nonumber\\
 &\le&
 C_1 \sum_{j\in\N} \Delta s | n_j^{m+1}-n_j^m|
 \nonumber\\
  &\le&
  C_1\Delta t\left(  \sum_{j\in\N} | n_j^{m}-n_{j-1}^m|  + \|p\|_\infty \sum_{j\in\N}  \Delta s n_j^m \right),\label{cota4}
 \end{eqnarray}
 and replacing \eqref{cota4} in \eqref{cota3}, we obtain from the mass conservation

\begin{equation}
\label{cota5}
    \begin{split}
       |\Delta^{+}n_0^{m+1}| & \le \left( 1 - \frac{\Delta t}{\Delta s} \right)  |\Delta^{+}n_0^{m}|
 + C_1\Delta t
  \left( \sum_{j\in\N}|\Delta^{+}n^m_{j-1}| + \|p\|_\infty \|n^0\|_1  \right)+ \Delta t \|p\|_\infty\|n^0\|_\infty.
    \end{split}
\end{equation}
 
Now, summing \eqref{cota2} and \eqref{cota5}, we deduce
\begin{equation}
 \sum_{j=0}^\infty|\Delta^{+}n_j^{m+1}|\le 
\left( 1 + 
 C_1\Delta t \right)
\sum_{j=0}^\infty|\Delta^{+}n_j^{m}|
+ C_2\Delta t.
\label{cota6}   
\end{equation}
with $C_2\coloneqq \|\partial_s p\|_\infty\|n^0\|_1+\|p\|_\infty C_1\|n^0\|_1+\|p\|_\infty\|n^0\|_\infty$.

Finally, proceeding recursively on $m$, we obtain

\begin{equation*}
\sum_{j=0}^\infty|\Delta^{+}n_j^m|\le (1+C_1\Delta t)^{m}\sum_{j=0}^\infty|\Delta^{+}n_j^{0}|+\frac{C_2}{C_1}\left((1+C_1\Delta t)^{m}-1\right),
\end{equation*}
and the estimate \eqref{BV_estimate} readily follows.
\end{proof}

\begin{remark}
The previous lemma is also valid for the case when the firing coefficient is of the form
\begin{equation*}
p(s,N)=\varphi(N)\chi_{\{s>\sigma(N)\}},
\end{equation*}
with $\varphi$ and $\sigma$ Lipschitz bounded functions and satisfying the analogous hypothesis on $\partial_N p$. Indeed, the inequality in \eqref{cota2} is replaced by

\begin{equation*}
    \begin{split}
    \sum_{j=1}^\infty|\Delta^{+}n_j^{m+1}|&\le
    \sum_{j=1}^\infty|\Delta^{+}n_j^{m}|
    + \frac{\Delta t}{\Delta s} |\Delta^{+}n_{0}^{m}|+\Delta t\varphi(N^m)\sum_{j\in\N}n^m_j\left|\chi_{\{s_{j+1}>\sigma(N^m)\}}-\chi_{\{s_j>\sigma(N^m)\}}\right|\\
    &\le \sum_{j=1}^\infty|\Delta^{+}n_j^{m}|
    + \frac{\Delta t}{\Delta s} |\Delta^{+}n_{0}^{m}|+\Delta t \|p\|_\infty n^m_{j_m}\\
    &\le \sum_{j=1}^\infty|\Delta^{+}n_j^{m}|
    + \frac{\Delta t}{\Delta s} |\Delta^{+}n_{0}^{m}|+\Delta t \|p\|_\infty\|n^0\|_\infty,
    \end{split}    
\end{equation*}
where $j_m\coloneqq\min\{j\in N\colon s_j\ge \sigma(N^m)\}=\left\lceil\frac{\sigma(N^m)}{\Delta s}+\frac{1}{2}\right\rceil$ and the rest of proof is analogous.
\end{remark}

Now we prove that the numerical approximation of the solution of \eqref{eq:model} equation $n(t,s)$, which is constructed by a simple piece-wise linear interpolation, has a limit when the time step $\Delta t$ and age step $\Delta s$ converge to $0$. For simplicity we assume that the initial data $n^0\in BV(\R^+)$ and with compact support.

\begin{lemma}
\label{lim-delta}
    Assume that $n^0\in BV(\R^+)$ is compactly supported and the coefficient $p$ satisfies the hypothesis of Theorem \ref{wellp-nodelay}. Consider the function $n_{\Delta t,\Delta s}\in\mathcal{C}\left([0,T],L^1(\R^+)\right)$ defined by
    \begin{equation}
    \label{n-interp}
      n_{\Delta t,\Delta s}(t,s)\coloneqq\frac{t^m-t}{\Delta t}\sum_{j\in\N}n_j^{m-1}\chi_{[s_{j-\frac{1}{2}},s_{j+\frac{1}{2}}]}(s)+\frac{t-t^{m-1}}{\Delta t}\sum_{j\in\N}n_j^{m}\chi_{[s_{j-\frac{1}{2}},s_{j+\frac{1}{2}}]}(s)\;\textrm{if}\; t\in[t^{m-1},t^{m}].  
    \end{equation}
    Then there exists a sub-sequence $(\Delta t_k,\Delta s_k)\to (0,0)$ when $k\to\infty$ and a function $\overline{n}(t,s)$ such that $n_{\Delta t_k,\Delta s_k}\to\overline{n}$ in $\mathcal{C}\left([0,T],L^1(\R^+)\right)$. Moreover, if we define the function $N_{\Delta t,\Delta s}(t)$ as the unique solution of the equation
 \begin{equation*}
    N(t)=\int_0^\infty p(s,N(t))n_{\Delta t,\Delta s}(t,s)\,\dd s,
   \end{equation*}
   then there exists $\overline{N}$ such that $N_{\Delta t_k,\Delta s_k}\to\overline{N}$ in $\mathcal{C}[0,T]$ and $\overline{N}$ is a solution of the equation
   \begin{equation}
   \label{lim-Nbar}
       \overline{N}(t)=\int_0^\infty p(s,\overline{N}(t))\overline{n}(t,s)\,\dd s.
   \end{equation}
\end{lemma}

\begin{proof}
  The idea is to apply the compactness criterion in $\mathcal{C}([0,T],L^1(\R^+))$ in order to extract a convergent sub-sequence of $n_{\Delta t,\Delta s}$ when $\Delta t$ and $\Delta s$ converge to $0$. From Lemma \ref{L1-Mass}, we deduce that
 \begin{equation*}
  \|n_{\Delta t,\Delta s}(t,\cdot)\|_1 = \|n^0\|_1,\qquad\forall t\in[0,T].
   \end{equation*}
We prove that sequence $n_{\Delta t,\Delta s}$ has a modulus of continuity in the $L^1$ in both variables. In the variable $s$ we have that for $t\in[t^{m-1},t^m]$ and $|h|<\varepsilon$ the following estimate holds
 \begin{equation*}
    \begin{split}
        \int_0^\infty|n_{\Delta t,\Delta s}(t,s+h)-n_{\Delta t,\Delta s}(t,s)|\,\dd s&\le \frac{t^m-t}{\Delta t}|h|\sum_{j\in\N}\Delta s n_j^{m-1}+\frac{t-t^{m-1}}{\Delta t}|h|\sum_{j\in\N}\Delta s n_j^m\\
        &\le \varepsilon\|n^0\|_1.
    \end{split}
\end{equation*}
Now we prove we have modulus of continuity in the variable $t$. Consider $t_1,t_2\in [0,T]$ and without loss of generality assume that $t_1,t_2\in [t^{m-1},t^{m}]$ for some $m\in\N$. Then from Lemma \ref{BV-lem} we have following estimate
    \begin{equation}
    \label{mod-cont-t}
    \begin{split}
        \int_0^\infty|n_{\Delta t,\Delta s}(t_1,s)-n_{\Delta t,\Delta s}(t_2,s)|\,\dd s&\le |t_1-t_2|\sum_{j\in\N}\Delta s \frac{|n^{m}_j-n^{m-1}_j|}{\Delta t}\\
        &\le |t_1-t_2|\left(\sum_{j\in\N} |n^{m-1}_j-n^{m-1}_{j-1}| +\Delta s p(s_j,N^{m-1})n^{m-1}_j\right) \\
        &\le |t_1-t_2|\left(C_T\,TV(n^0)+\|p\|_\infty\|n^0\|_1\right),
    \end{split}  
    \end{equation}
    thus we have the modulus of continuity in time. Since $n^0$ has its support contained in some interval $[0,R]$, there exists $K\in\N$ such that $n^0_j=0$ for $j\ge K$ and $s_K=(K-\frac{1}{2})\Delta s\ge R$. From the numerical scheme we deduce that $n^m_j=0$ for $j\ge K+M$, which implies that $n_{\Delta t,\Delta s}$ vanishes for 
    \begin{equation*}
    s\ge s_{K+M}=(K+M-\frac{1}{2})\Delta s\ge R+M\Delta{t}\,\frac{\Delta s}{\Delta t}\ge R+T,
    \end{equation*}
    so $n_{\Delta t,\Delta s}$ has also compact support. From the estimates on the modulus of continuity, we can apply the compactness criterion in $\mathcal{C}([0,T],L^1(\R^+))$ in order to extract a convergent sub-sequence of $n_{\Delta t,\Delta s}$ when $\Delta t$ and $\Delta s$ converge to $0$.
    
    Observe now that $N_{\Delta t,\Delta s}\in\mathcal{C}[0,T]$. Indeed from Lemma \ref{fixedpoint-itm}, for each $t\in[0,T]$ we get that
\begin{equation*}
    N_{\Delta t,\Delta s}(t)=\psi(n_{\Delta t,\Delta s}),
 \end{equation*}
    and from the continuity of $\psi$ and $n_{\Delta t,\Delta s}\in\mathcal{C}([0,T],L^1(\R^+)$ we obtain that $N_{\Delta t,\Delta s}\in\mathcal{C}[0,T]$. Since $\|n_{\Delta t,\Delta s}(t,\cdot)\|_1=\|n^0\|_1,$ for all $t\in [0,T]$ we have the following estimate
 \begin{equation*}
    \|N_{\Delta t,\Delta s}\|_\infty\le \|p\|_\infty\|n^0\|_1,
 \end{equation*}
    so that $N_{\Delta t,\Delta s}$ is uniformly bounded.
    
    We now prove that the family $N_{\Delta t,\Delta s}$ is equicontinuous. For $t_1,t_2\in [0,T]$ we deduce from Lemma \eqref{fixedpoint-itm} and estimate \eqref{mod-cont-t} the following inequality
    \begin{equation*}
        \begin{split}
           \left|N_{\Delta t,\Delta s}(t_1)-N_{\Delta t,\Delta s}(t_2)\right|&=|\psi(n_{\Delta t,\Delta s})(t_1)-\psi(n_{\Delta t,\Delta s})(t_2)|\\
           &\le C\int_0^\infty|n_{\Delta t,\Delta s}(t_1,s)-n_{\Delta t,\Delta s}(t_2,s)|\,\dd s\\
           &\le |t_1-t_2|\left(C_T\,TV(n^0)+\|p\|_\infty\|n^0\|_1\right),
        \end{split}
    \end{equation*}
    where $C$ is a constant independent of $\Delta t$ and $\Delta s$. Therefore the family $N_{\Delta t,\Delta s}$ is equicontinuous in $[0,T]$ and we can extract a convergent sub-sequence by applying Arzelà-Ascoli Theorem. Finally, by passing to the limit in $\Delta t$ and $\Delta s$ we obtain Equation \eqref{lim-Nbar}.
\end{proof}

With the previous lemmas, we are now ready to the prove the following theorem on convergence of the numerical scheme for the \eqref{eq:model} equation.

\begin{theorem}[Convergence of the numerical scheme]
\label{conv-num}
Assume that $n^0\in BV(\R^+)$ is compactly supported and the coefficient $p$ satisfies the hypothesis of Theorem \ref{wellp-nodelay}. Then for all $T>0$, the numerical approximation \eqref{n-interp} given by the upwind scheme converges to the unique weak solution $n\in\mathcal{C}([0,T],L^1(\R^+))$ of the \eqref{eq:model} equation.
\end{theorem}

\begin{proof}
    Consider the functions $n_{\Delta t,\Delta s}$ and $N_{\Delta t,\Delta s}$ defined in Lemma \ref{lim-delta}. From this result we get Equation \eqref{lim-delta}. Now we take $\varphi\in\mathcal{C}^1_c([0,T)\times[0,\infty))$ a test function. If we multiply Equation \eqref{finit-diff} by $\varphi_j^m\coloneqq\varphi(t^m,s_j)$ and compute the discrete integral we get
    \begin{equation}
    \label{discrete}
      \sum_{m=0}^{M}\sum_{j\in\N}\Delta s(n_j^{m+1}-n_j^m)\varphi_j^m+\sum_{m=0}^{M}\sum_{j\in\N}\Delta t(n_j^m-n_{j-1}^m)\varphi_j^m+\sum_{m=0}^{M}\sum_{j\in\N}\Delta t\Delta s p(s_j,N^m)n_j^m\varphi_j^m=0,  
    \end{equation}
    We study each term of Equation \eqref{discrete}. From summation by parts we have the following inequality for the first term
    \begin{equation*}
       \sum_{m=0}^{M}\sum_{j\in\N}\Delta s(n_j^{m+1}-n_j^m)\varphi_j^m =-\sum_{j\in\N}\Delta s\varphi^0_j n^0_j-\sum_{m=1}^{M}\sum_{j\in\N}\Delta s n^{m}_{j}(\varphi^{m}_j-\varphi^{m-1}_j),
    \end{equation*}
    thus applying Lemma \ref{lim-delta}, we get the following limit when $(\Delta t,\Delta s)\to 0$
    \begin{equation*}
        \sum_{m=0}^{M}\sum_{j\in\N}\Delta s(n_j^{m+1}-n_j^m)\varphi_j^m \to -\int_0^\infty\varphi(0,s)n^0(s)\,\dd s-\int_0^T\int_0^\infty \overline{n}(t,s)\partial_t\varphi(t,s)\,\dd s\,\dd t.
    \end{equation*}
    Similarly, for the second term of Equation \eqref{discrete} the following equality holds
    \begin{equation*}
          \sum_{m=0}^{M}\sum_{j\in\N}\Delta t(n_j^m-n_{j-1}^m)\varphi_j^m =-\Delta t \varphi_0^0 n^0_0 -\sum_{m=1}^{M}\Delta t \varphi_{0}^m N^m - \sum_{m=0}^{M}\sum_{j\in\N}\Delta t n^m_{j}(\varphi^m_{j}-\varphi^m_{j-1}), 
    \end{equation*}
    and by passing to the limit in $(\Delta t,\Delta s)$ we get
      \begin{equation*}
       \sum_{m=0}^{M}\sum_{j\in\N}\Delta t(n_j^m-n_{j-1}^m)\varphi_j^m  \to -\int_0^T\varphi(t,0)\overline{N}(t)\,\dd t-\int_0^T\int_0^\infty \overline{n}(t,s)\partial_s\varphi(t,s)\,\dd s\,\dd t,
    \end{equation*}
    and in the same way
    \begin{equation*}
    \sum_{m=0}^{M}\sum_{j\in\N}\Delta t\Delta s\,p(s_j,N^m)n_j^m\varphi_j^m\to \int_0^T\int_0^\infty p(s,\overline{N}(t))\overline{n}(t,s)\varphi(t,s)\,\dd s\,\dd t.
    \end{equation*}
    Therefore $\overline{n}(t,s)$ is the weak solution of the \eqref{eq:model} equation.
\end{proof}

\section{Distributed Delay Model (DDM)}
\label{DDM}

\subsection{Well-posedness of DDM revisited}
In this subsection we revisit well-posedness of the equations \eqref{eq:delay} and \eqref{eq:delay}. We essentially follow the proof in \cite{chevallier2015mean} with some slight modifications inspired by the Section \ref{ITM}, but for the sake of completeness we also include the proof as in the \eqref{eq:model} equation. The result on well-posedness is the base idea of the corresponding numerical scheme and it also improves the proof of \cite{pakdaman2009dynamics,canizo2019asymptotic} by extending existence and uniqueness for more general types of firing coefficients $p$.

\begin{theorem}
\label{wellp-delay}
      Assume that $p\in W^{1,\infty}(\R^+\times\R^+)$ and $\alpha\in L^1(\R^+)$ is bounded. Then for a nonnegative $n^0\in L^1(\R^+)$,  Equation \eqref{eq:delay} has a unique solution $n\in \mathcal{C}\left([0,\infty),L^1(\R^+)\right)$ and $N,X\in\mathcal{C}[0,\infty)$.
\end{theorem}

For the proof we need the following lemma, which is the analogous of Lemma \ref{fixedpoint-itm}.

\begin{lemma}
\label{fixedpoint-delay}
    Consider a nonnegative function $n\in\mathcal{C}\left([0,T],L^1(\R^+)\right)$. Then for $T$ small enough, the exists a unique solution $X\in\mathcal{C}([0,T])$ of the integral equation
    \begin{equation}
    \label{fixed-X}
       X(t)=F(X(t))\coloneqq \int_0^t\int_0^\infty\alpha(t-\tau)p(s,X(\tau))n(\tau,s)\,\dd s\,\dd \tau, 
    \end{equation}
    that we call $X\coloneqq\psi(n)$, where the map $\psi\colon \mathcal{C}\left([0,T],L^1(\R^+)\right)\to\mathcal{C}([0,T])$ satisfies the following estimate
    \begin{equation}
    \label{psi-lip-d}
        \|\psi(n^1)-\psi(n^2)\|_\infty\le \tfrac{T\|\alpha\|_\infty\|p\|_\infty}{1-T\|\alpha\|_\infty\|\partial_X p\|_\infty\|n^0\|_1}\sup_{t\in[0,T]}\|n^1(t,\cdot)-n^2(t,\cdot)\|_1,
    \end{equation}
    for nonnegative integrable functions $n^1,n^2$ with $\|n^1(t,\cdot)\|_1=\|n^2(t,\cdot)\|_1=\|n^0\|_1$ for all $t\in[0,T]$.
\end{lemma}

\begin{proof}
    Observe that the map $F\colon\mathcal{C}[0,T]\to\mathcal{C}[0,T]$ satisfies the following estimate
 \begin{equation*}
    \|F(X_1)-F(X_2)\|_\infty\le T\|\alpha\|_\infty\|\partial_X p\|_\infty\|n\|\|X_1-X_2\|_\infty,
 \end{equation*}
    with $\|n\|=\sup_{t\in[0,T]}\|n(t,\cdot)\|_1$. Hence for $T>0$ such that $T\|\alpha\|_\infty\|\partial_X p\|_\infty\|n\|<1$ the map $F$ is a contraction and then the map $\psi$ is well-defined.    Let $n^1,n^2\in\mathcal{C}([0,T],L^1(\R^+))$ with $\|n^1(t,\cdot)\|_1=\|n^2(t,\cdot)\|_1=\|n^0\|_1$ for all $t\in[0,T]$. Then we have the following inequality
    \begin{equation*}
        \begin{split}
            |X_1-X_2|(t)&\le \int_0^t\int_0^\infty\alpha(t-\tau)\left(|p(s,X_1(\tau))-p(s,X_2(\tau))|n^1(\tau,s) +  p(s,X_2(\tau))|n^1-n^2|(\tau,s)\right)\dd s\dd \tau\\
            &\le T\|\alpha\|_\infty\|\partial_X p\|_\infty\|n^0\|_1\|X_1-X_2\|_\infty+T\|\alpha\|_\infty\|p\|_\infty\sup_{t\in[0,T]}\|n^1(t,\cdot)-n^2(t,\cdot)\|_1,
        \end{split}
    \end{equation*}
    and therefore for $T>0$ such that $T\|\alpha\|_\infty\|\partial_X p\|_\infty\|n^0\|_1<1$ we get
 \begin{equation*}
    \|X_1-X_2\|_\infty\le\frac{T\|\alpha\|_\infty\|p\|_\infty}{1-T\|\alpha\|_\infty\|\partial_X p\|_\infty\|n^0\|_1}\sup_{t\in[0,T]}\|n^1(t,\cdot)-n^2(t,\cdot)\|_1,
 \end{equation*}
    and estimate \eqref{psi-lip-d} holds.
\end{proof}

With this lemma, we continue the proof of the Theorem \ref{wellp-delay}.

\begin{proof}
    Consider $T>0$ and a given nonnegative $X\in\mathcal{C}[0,T]$. Like in the proof of Theorem \ref{wellp-nodelay}, we define $n[X]\in\mathcal{C}([0,T],L^1(\R^+))$ as the weak solution of the linear equation
\begin{equation*}
    \begin{cases}
    \partial_tn+\partial_s n+p(s,X(t))n = 0,  & t>0,\: s>0, \\
    n(t,s=0)=N(t)=\int_0^{+\infty}p(s,X(t))n(t,s) \dd s,& t>0,\\
    n(0,s)=n^0(s)\geq0,& s\geq 0.
    \end{cases}
\end{equation*}
    which can be expressed trough the method of characteristics
 \begin{equation*}
    n(t,s)=n^0(s-t)e^{-\int_0^t p(s-t+t',X(t'))\dd t'}\chi_{\{s>t\}}+N(t-s)e^{-\int_0^s p(s',X(t-s+s'))\dd s'}\chi_{\{t>s\}}.
 \end{equation*}
    From the mass conservation property we know that $\|n[X](t)\|_1=\|n^0\|_1$.
    
    On the other hand, for a given $n\in\mathcal{C}([0,T],L^1(\R^+))$ and for $T>0$ small enough, let $\psi\colon \mathcal{C}\left([0,T],L^1(\R^+)\right)\to\mathcal{C}[0,T]$ the unique solution $X\in\mathcal{C}[0,T]$ of Equation \eqref{fixed-X}. Under this setting, we get a solution of the nonlinear equation \eqref{eq:delay} if only if we find $X\in\mathcal{C}[0,T]$ such that is a solution of the equation
\begin{equation*}
X=H(X)\coloneqq\psi\left(n[X]\right),
\end{equation*}
    where $H\colon\mathcal{C}[0,T]\to\mathcal{C}[0,T]$.
    
    We assert that for $T$ small enough the map $H$ is a contraction following the proof Theorem \ref{wellp-nodelay}. For nonnegative functions $X_1,X_2\in\mathcal{C}[0,T]$, we get from the method of characteristics we get the following inequality
\begin{equation*}
    \int_0^\infty |n[X_1]-n[X_2]|(t,s)\dd s\le A_1+A_2+A_3,    
\end{equation*}
    where $A_1,A_2,A_3$ are given by
    \begin{eqnarray*}
        A_1&=&\int_0^\infty n^0(s)\left|e^{-\int_0^t p(s+t',X_1(t'))\dd t'}-e^{-\int_0^t p(s+t',X_2(t'))\dd t'}\right|\,\dd s,\\
        A_2&=&\int_0^t |N_1-N_2|(t-s) e^{-\int_0^s p(s',X_1(t-s+s'))\dd s'}\dd s,\\
        A_3&=&\int_0^t N_2(t-s)\left|e^{-\int_0^s p(s',X_1(t-s+s'))\dd s'}-e^{-\int_0^s p(s',X_2(t-s+s'))\dd s'} \right|\dd s.
    \end{eqnarray*}
    We proceed by estimating each term. For simplicity we write $n_1=n[X_1],n_2=n[X_2]$, so that for $A_1$ we have
 \begin{equation*}
    A_1\le \int_0^\infty n^0(s)\int_0^t|p(s+t',X_1(t'))-p(s+t',X_2(t'))|\,\dd t'\,\dd s \le T\|n^0\|_1 \|\partial_N p\|_\infty \|X_1-X_2\|_\infty,
 \end{equation*}
    while for $A_2$ we have
    \begin{equation*}
        \begin{split}
           A_2 & \le \int_0^t\int_0^\infty|p(s',X_1(s))-p(s',X_2(s))|n_1(s,s')\dd s'\,\dd s+\int_0^t\int_0^\infty p(s',X_2(s))|n_1-n_2|(s,s')\dd s'\,\dd s\\
           &\le T\|\partial_X p\|\|n^0\|_1 + T\|p\|_\infty\sup_{t\in[0,T]}\|n^1(t,\cdot)-n^2(t,\cdot)\|_1,
        \end{split}
    \end{equation*}
    and for $A_3$ we get
    \begin{equation*}
    \begin{split}
        A_3 &\le \|p\|_\infty\|n^0\|_1 \int_0^t\int_0^s|p(s',X_1(t-s+s'))-p(s',X_2(t-s+s'))|\dd s'\,\dd s\\
        &\le\frac{T^2}{2} \|p\|_\infty\|n^0\|_1\|\partial_N p\|_\infty\|X_1-X_2\|_\infty. 
    \end{split}
    \end{equation*}
    By combining these estimates, we get for $T<\frac{1}{\|p\|_\infty}$  the following inequality
    \begin{equation}
    \label{nX1-nX2}
       \sup_{t\in [0,T]}\|n_1(t,\cdot)-n_2(t,\cdot)\|_1\le\tfrac{T}{1-T\|p\|_\infty}\left(2\|n^0\|_1\|\partial_N p\|_\infty+\frac{T}{2}\|p\|_\infty\|n^0\|_1\|\partial_N p\|_\infty\right)\|X_1-X_2\|_\infty.
    \end{equation}
    From the mass conservation property, we get from \eqref{psi-lip-d} and \eqref{nX1-nX2} that the following estimate holds for $H$
    \begin{equation*}
        \begin{split}
          \|H(X_1)-H(X_2)\|_\infty&\le \tfrac{T^2\|p\|_\infty^2\|\alpha\|_\infty}{(1-T\|\alpha\|_\infty\|n^0\|_1\|\partial_X p\|_\infty)(1-T\|p\|_\infty)}\left(2\|n^0\|_1\|\partial_N p\|_\infty+\frac{T}{2}\|p\|_\infty\|n^0\|_1\|\partial_N p\|_\infty\right)\|X_1-X_2\|_\infty.  
        \end{split}
    \end{equation*}
    For $T>0$ small enough we get a unique fixed point of $H$ by contraction principle, implying the existence of a unique solution $n\in \mathcal{C}\left([0,T],L^1(\R^+)\right)$ with $N,X\in\mathcal{C}[0,T]$ of Equation \eqref{eq:delay}. In order to extend the solution for all times, we split the integral involving the distributed delay
   \begin{equation*}
    X(t)=\int_0^T \alpha(t-\tau)N(\tau)\dd \tau+\int_T^t\alpha(t-\tau)N(\tau)\dd \tau.
   \end{equation*}
    Since the first term is already known and $T$ is independent of the initial data, we can reapply the argument on existence to have the solution of Equation \eqref{eq:delay} defined for all $t\in[T,2T]$. By iterating the splitting argument involving $X(t)$, we conclude that the solution of Equation \eqref{eq:delay} is defined for all $t>0$.
\end{proof}

\begin{remark}
    As in Theorem \ref{wellp-nodelay}, the regularity of the coefficient $p$ is not fundamental for the proof and Theorem \ref{wellp-nodelay} is still valid for coefficients $p$ of the form
    \begin{equation*}
    p(s,X)=\varphi(X)\chi_{\{s>\sigma(X)\}},      
    \end{equation*}
    with $\varphi$ and $\sigma$ Lipschitz bounded functions. Moreover, the proof can be adapted when $p$ is not necessarily bounded, but the continuous solution might not be defined for all $t>0$. We can extend the solution as long as $X(t)<\infty$.

    For example, consider $p(s,X)=X^2+1$ and $\alpha(t)=e^{-t}$, so that from the mass-conservation in the \eqref{eq:delay} equation, we get
    \begin{equation*}
    N(t)=X^2(t)+1,\:X(t)=\int_0^t e^{-(t-s)}(X^2(s)+1)\,ds,
    \end{equation*}
    and by taking derivative for $X(t)$, we deduce the following differential equation
    \begin{equation*}
    X'(t)=X^2(t)-X(t)+1,\: X(0)=0.
    \end{equation*}
    Since the solution of this equation blows up for a finite $T>0$, we conclude that the solution of \eqref{eq:delay} cannot be  continuously extended for all $t>0$. In particular, one can show that $n(t,\cdot)\rightharpoonup\delta_0$, a Dirac mass, when $t\to T$.
    \end{remark}

\subsection{Numerical scheme for DDM}
In this section, we present the numerical analysis of the \eqref{eq:delay} equation. Using the same discretization and notation introduced in Section \ref{ITM}, Equation \eqref{eq:delay} can be solved numerically using the following explicit scheme

\begin{equation}\label{eq:nn}
    n_j^{m+1}=n_j^m-\frac{\Delta t}{\Delta s}(n_{j}^m-n_{j-1}^m)-\Delta t p(s_j,X(t^m))n_j^m,\quad j\in\N.
\end{equation}
In particular for $j=1$
\begin{equation*} 
    n_1^{m+1}=n_1^m-\frac{\Delta t}{\Delta s}(n_{1}^m-N(t^m))-\Delta t p(s_1,X(t^m))n_1^m, 
\end{equation*}
where 
\begin{equation}\label{eq:Nm}
    N(t^m)=n(t^m,0)\approx \Delta s \sum_{j\in\N} p(s_j,X(t^m))n_j^m,
\end{equation}
and using a trapezoidal quadrature rule we have
\begin{equation}\label{eq:Xm}
    X(t^m)=\int_0^{t^m} N(t^m-s)\alpha(s)\,\dd s\approx \frac{\Delta t}{2} \sum_{k=0}^m N(t^k)\alpha_{m-k},
\end{equation}
where $\alpha_{m}=\alpha(t^m)$ and we denote $\|\alpha\|_1\coloneqq\sum_{k\in\N}\Delta t|\alpha_k|$. If we set $X^m\coloneqq X(t^m)$ and $N^m\coloneqq N(t^m)$, we can combine the equations \eqref{eq:Nm} and \eqref{eq:Nm} and to solve for $X^m$

\begin{equation}\label{eq:findX}
    X^m=\frac{\Delta t}{2} \left( \Delta s \sum_{j\in\N}p(s_j,X^m)n_j^m \alpha_0 + \sum_{k=0}^{m-1} N^k\alpha_{m-k} \right),
\end{equation}
and then we obtain $N^m$ from \eqref{eq:Nm}. In that regard, the numerical method is given as follows.

\begin{Alg}{DDM numerical scheme }\label{Alg2}
 \begin{itemize}
 \item[] {Input: Approximate initial data $\{n_j^0\}_{j\in\N}$}
 \item[] \begin{equation*}
    X^0\leftarrow 0,\quad N^{0}\leftarrow \Delta s \sum_{j\in\N} p(s_j,X^0)n_j^{0}.  
 \end{equation*}
 \item[] Choose $\Delta t$ such that 
 \begin{equation}
 \label{CFL2}
      \Delta t < \min\left\{\left(\frac{1}{\Delta s}+\|p\|_\infty\right)^{-1},\frac{2}{\alpha_0\|\partial_Xp\|_{\infty}\|n^0\|_1}\right\}
 \end{equation}
 \item[] {\bf For} $m\in\N_0$
 \item[] \hspace{0.5cm}{\bf Do} $j\in\N$, 
    \begin{equation*}
    n_j^{m+1} \leftarrow
    \begin{cases}
      n_1^m-\frac{\Delta t}{\Delta s}(n_{1}^m-N^m)-\Delta t p(s_1,N^m)n_1^m & j=1\\
      n_j^m-\frac{\Delta t}{\Delta s}(n_{j}^m-n_{j-1}^m)-\Delta t p(s_j,N^m)n_j^m & j>1
     \end{cases}
    \end{equation*}
 \item[] \hspace{0.5cm}{\bf end} 
 \item[] \hspace{0.5cm} Solve for $X^{m+1}$ 
 \begin{equation}\label{eqfijo-X}
     X^{m+1}=\frac{\Delta t}{2} \left( \Delta s \sum_{j\in\N}p(s_j,X^{m+1})n_j^{m+1} \alpha_0 + \sum_{k=0}^{m} N^k\alpha_{m-k}. \right)
 \end{equation}
 \item[]
 \begin{equation*}
    N^{m+1}\leftarrow \Delta s \sum_{j\in\N} p(s_j,X^{m+1} )n_j^{m+1} 
 \end{equation*}
 \item[] {\bf end} 
 \item[] { Output:  approximate solution $\{n_j^{m+1}\}_{j\in\N}$  and $X^{m+1},\,N^{m+1}$ at time $t^{m+1}=(m+1)\Delta t$}.
 \end{itemize} 
\end{Alg}

Analogously to the numerical scheme for the \eqref{eq:model} equation. The solution of Equation \eqref{eqfijo-X} for $X^{m+1}$ can be solved with different numerical methods. Unlike the instantaneous transmission model, 
there is no restriction on the coefficient $p$ to have unique solution of Equation \eqref{eq:delay}. Hence, by following the idea of Lemma \ref{fixedpoint-delay} and the contraction principle, the solution of Equation \eqref{eqfijo-N} $N^{m+1}$ can be approximated in terms of $N^{m}$ through the following formula if $\Delta t$ is small enough:
\begin{equation*}
X^{m+1}=\frac{\Delta t}{2} \left( \Delta s \sum_{j\in\N}p(s_j,X^{m})n_j^{m+1} \alpha_0 + \sum_{k=0}^{m} N^k\alpha_{m-k} \right).
\end{equation*}
For simplicity in the estimates we assume that we can compute the solution of Equation \eqref{eqfijo-X} exactly, but the results remain valid if we take into account a specific method to get an approximation.

In order to prove the convergence of the upwind scheme we follow the ideas of the previous subsection on well-posedness and we prove a BV-estimate that will be the crucial in the analysis. For simplicity we assume that initial data $n^0$ is compactly supported, but the theoretical results still hold when the initial data $n^0$ vanishes at infinity.

As in the case of the \eqref{eq:model} equation, we get the corresponding lemmas on $L^1$ and $L^\infty$ norms.

 \begin{lemma}{(\textbf{$L^{1}$}-norm)} Numerical approximation obtained with the Algorithm \ref{Alg2} satisfies 
 \begin{equation*}
     \|n^{m}\|_1:=\sum_{j\in\N}\Delta s n_j^m=\|n^{0}\|_1, \quad m\in\N.
 \end{equation*}
 \end{lemma}
 \begin{proof}
The proof is the same as Lemma \ref{L1-Mass}.
 \end{proof}

\begin{lemma}{(\textbf{$L^{\infty}$}-norm)}
 Assume that $n^0\in L^\infty(\R^+)$ is nonnegative, then under the condition \eqref{CFL2}, Equation \eqref{eqfijo-X} has a unique solution and the numerical solution obtained by Algorithm \ref{Alg2}, satisfies the following estimates
 \begin{eqnarray*}
 0\le n_j^m\le \|n^0\|_{\infty}, &  0\le X^m\le \|p\|_{\infty}\|n^0\|_1\|\alpha\|_1, &  0\le N^m\le \|p\|_{\infty}\|n^0\|_1\quad\textrm{for all}\:j,m\in\N.
 \end{eqnarray*}
\end{lemma}

\begin{proof}
First, observe that $X^0=0$ and $N^0=\Delta s\sum_{j\in\N} p(s_j,0)n_j^0 \le \|p\|_\infty $. Now, using the CFL condition \eqref{CFL2} we have 
\begin{equation*} 
    0\le n_j^{1}=n_j^0 \left(1-\Delta t \left(\frac{1}{\Delta s}+p(s_j,0) \right) \right) +\frac{\Delta t}{\Delta s} n_{j-1}^0\le \|n^0\|_{\infty},
\end{equation*}
for $j\geq1$. In order to proof existence of $X^1$ which satisfy equation \eqref{eq:findX}, we consider the function 
\begin{equation*}
F(X)=\frac{\Delta t}{2} \left( \Delta s \sum_{j\in\N}p(s_j,X)n_j^1 \alpha_0 + N^0\alpha_{1} \right),
\end{equation*}
and observe that
\begin{eqnarray*}
&0\le F(X)\le\|p\|_{\infty}\|n^0\|_1\|\alpha\|_1,& \text{for all } X\ge 0,\\
&|F(X_1)-F(X_2)|\le \frac{\Delta t}{2}\|n^0\|_1\alpha_0\|\partial_Xp\|_{\infty}|X_1-X_2|,&  \text{for all } X_1,X_2\ge 0.
\end{eqnarray*}
From Condition \eqref{CFL2} we have
\begin{equation*}
\frac{\Delta t}{2}\alpha_0\|\partial_Xp\|_{\infty}\|n^0\|_1<1,
\end{equation*}
so that from contraction principle that there exists a unique $X^1$ such that
\begin{equation*}
X^1=F(X^1)=\frac{\Delta t}{2} \left( \Delta s \sum_{j\in\N}p(s_j,X^1)n_j^1 \alpha_0 + N^0\alpha_{1} \right),
\end{equation*}
and we have that $0\le X^1\le \|p\|_{\infty}\|n^0\|_1 \|\alpha\|_1$. Moreover we get the following estimate
\begin{equation*}
    N^1:= \Delta s \sum_{j\in\N} p(s_j,X^1)n_j^1\le \|p\|_{\infty} \|n^0\|_{1},
\end{equation*}
and we conclude the desired result by iterating this argument for all $m\in\N$.
\end{proof}

We note that in Condition \eqref{CFL2}, the first term in right-hand side corresponds to the CFL condition of the explicit scheme and the second term ensures that the Equation \eqref{eqfijo-X} has a unique solution so that $X^m$ is well-defined.

Next, we proceed with the corresponding BV-estimate that gives the necessary compactness to prove the convergence of the numerical scheme.
 
\begin{lemma}{(\textbf{BV}-estimate)}
\label{BV-lem-d}
Assume that $TV(n^0)<\infty$ and the CFL condition \eqref{CFL}. Then there exist constants $C_1,C_2>0$ (depending only $p,\alpha$ and the norms of $n^0$) such that for $m\in\N$ we have
\begin{equation}
    TV(n^m)\le e^{C_1T} TV(n^0)+C_2(e^{C_1T}-1),
    \label{BV_estimate-d}
\end{equation}
with  $T=m\Delta t$ and $TV(n^m)= \displaystyle\sum_{j=0}^\infty|n_{j+1}^{m}-n_{j}^{m}|$.
\end{lemma}
\begin{proof}
using notation $\Delta^{+}n_j^m=n_{j+1}^{m}-n_{j}^{m}$, we have
\begin{equation*}
    \begin{split}
      \Delta^{+}n_j^{m+1}&=  \Delta^{+}n_j^{m}-\frac{\Delta t}{\Delta s}\left(\Delta^{+}n_j^{m}-\Delta^{+}n_{j-1}^{m}\right)-\Delta t p(s_{j+1},X^m)\Delta^{+}n_j^{m}
       -\Delta t n_j^m(p(s_{j+1},X^m)-p(s_{j},X^m)).
    \end{split}
\end{equation*}
Next, from the CFL condition \eqref{CFL} we obtain
\begin{equation*}
     |\Delta^{+}n_j^{m+1}| \le \left(1-\Delta t\left( \frac{1}{\Delta s}+p(s_{j+1},X^m) \right) \right)|\Delta^{+}n_j^{m}|
     +\frac{\Delta t}{\Delta s}|\Delta^{+}n_{j-1}^{m}|+\Delta t\Delta s\|\partial_s p\|_\infty n_j^m.
\end{equation*}
By summing over all $j\ge 1$ we deduce that
\begin{equation}
\label{Delta+nm}
\sum_{j=1}^\infty|\Delta^{+}n_j^{m+1}|\le
\sum_{j=1}^\infty|\Delta^{+}n_j^{m}|
+ \frac{\Delta t}{\Delta s} |\Delta^{+}n_{0}^{m}|+\Delta t\|\partial_s p\|_\infty\|n^0\|_1. 
\end{equation}
We now take into account the boundary term. From the numerical scheme we have
\begin{equation*}
\label{Delta+n0-d}
\begin{split}
     |\Delta^{+}n_0^{m+1}| & = |n_1^{m+1}-N^{m+1}|\\
     & \le \left(1-\frac{\Delta t}{\Delta s}\right)|n_1^m-N^m|+  |N^{m+1}- N^m|+\Delta t p(s_1,N^m) n^m_1\\
     &\le \left(1-\frac{\Delta t}{\Delta s}\right)|\Delta^+ n_0^m|+  |N^{m+1}-N^m|+\Delta t \|p\|_\infty \|n^0\|_\infty.
\end{split} 
\end{equation*}

Next we estimate the second term of the last inequality. First note that
\begin{equation*}
    \begin{split}
      N^{m+1}- N^{m}& =\Delta s \sum_{j\in\N}p(s_j,X^{m+1})n_j^{m+1}-\Delta s \sum_{j\in\N} p(s_j,X^{m})n_j^{m}\\
    &=\Delta s \sum_{j\in\N}(p(s_j,X^{m+1})-p(s_j,X^{m}))n_j^{m+1}+\Delta s \sum_{j\in\N}p(s_j,X^m)(n_j^{m+1}-n_j^m),
    \end{split}
\end{equation*}
thus we have
\begin{equation}
\label{N1-N2-d}
    \begin{split}
      |N^{m+1}- N^{m}|&\le \|\partial_X p\|_\infty \|n^0\|_1 |X^{m+1}-X^{m}|+\|p\|_\infty\sum_{j\in\N}\Delta s|n^{m+1}_j-n^m_j|\\
      &\le  \|\partial_X p\|_\infty \|n^0\|_1 |X^{m+1}-X^{m}| + \|p\|_\infty\Delta t\left(  \sum_{j\in\N} | n_j^{m}-n_{j-1}^m|  + \|p\|_\infty \sum_{j\in\N}  \Delta s n_j^m \right)\\
      &\le  \|\partial_X p\|_\infty \|n^0\|_1 |X^{m+1}-X^{m}| + \|p\|_\infty\Delta t  \sum_{j\in\N} |\Delta^+ n^m_{j-1}| + \Delta t\|p\|_\infty^2\|n^0\|_1.
    \end{split}
\end{equation}
Similarly for $|X^{m+1}-X^{m}|$ we get by using Equation \eqref{eq:findX}
\begin{equation*}
    \begin{split}
       |X^{m+1}-X^m|&\le \frac{\Delta t}{2}\alpha_0\|\partial_X p\|_\infty\|n^0\|_1 |X^{m+1}-X^{m}|+\frac{\Delta t}{2}\alpha_0\|p\|_\infty\sum_{j\in\N}\Delta s|n^{m+1}_j-n^m_j|\\
       &\qquad +\frac{\Delta t}{2}\left|\sum_{k=0}^m N^k\alpha_{m+1-k}-\sum_{k=0}^{m-1} N^k\alpha_{m-k}\right|,
    \end{split}
\end{equation*}
so we get
\begin{equation*}
    \begin{split}
       |X^{m+1}-X^m|&\le \frac{\Delta t}{2}\alpha_0\left(\|\partial_X p\|_\infty\|n^0\|_1 |X^{m+1}-X^{m}|+\Delta t\|p\|_\infty\sum_{j\in\N} |\Delta^+ n^m_{j-1}| +\Delta t\|p\|_\infty^2\|n^0\|_1\right)\\
       &\qquad +\frac{\Delta t}{2} N^m\alpha_1+\frac{\Delta t}{2}\|p\|_\infty\|n^0\|_1\sum_{k=0}^{m}|\alpha_{m+1-k}-\alpha_{m-k}|,
    \end{split}
\end{equation*}
and therefore we have the following estimate
\begin{equation}
\begin{split} \label{X1-X2-d}
     |X^{m+1}-X^m| &\le \frac{\Delta t}{2(1-\frac{\Delta t}{2}\alpha_0\|\partial_X p\|_\infty\|n^0\|_1)}\left(\Delta t\alpha_0\|p\|_\infty\sum_{j\in\N} |\Delta^+ n^m_{j-1}| +\Delta t\alpha_0\|p\|_\infty^2\|n^0\|_1 \right. \\
     &\left.\phantom{\sum_{j\in\N}}+
    \|p\|_\infty\|n^0\|_1\|\alpha\|_\infty+\|p\|_\infty\|n^0\|_1 TV(\alpha)\right).
\end{split}
\end{equation}
By plugging \eqref{X1-X2-d} into \eqref{N1-N2-d} we obtain for $\Delta t$ small
\begin{equation*}
|N^{m+1}-N^m|\le A_1\Delta t\sum_{j\in\N}|\Delta^+ n^m_{j-1}|+B_1\Delta t,
\end{equation*}
where $A_1,B_1>0$ are constants depending on $p,\alpha$ and the norms of $n^0$. So that the boundary term is estimated as follows
\begin{equation}
\label{Delta+n0}
\begin{split}
     |\Delta^{+}n_0^{m+1}|
     &\le \left(1-\frac{\Delta t}{\Delta s}\right)|\Delta^+ n_0^m|+  A_1\Delta t\sum_{j\in\N}|\Delta^+ n^m_{j-1}|+B_2\Delta t,
\end{split}
\end{equation}
and by adding \eqref{Delta+n0} with \eqref{Delta+nm}, we obtain
\begin{equation}
\sum_{j=0}^\infty|\Delta^{+}n_j^{m+1}|\le
(1+C_1\Delta t)\sum_{j=0}^\infty|\Delta^{+}n_j^{m}|+C_2\Delta t,
\end{equation}
where $C_1,C_2>0$ constants depending only on $p,\alpha$ and the norms of $n^0$.

Finally, proceeding recursively on $m$, we obtain
\begin{equation*}
 \sum_{j=0}^\infty|\Delta^{+}n_j^m|\le (1+C_1\Delta t)^{m}\sum_{j=0}^\infty|\Delta^{+}n_j^{0}|+\frac{C_2}{C_1}\left((1+C_1\Delta t)^{m}-1\right),
\end{equation*}
and the estimate \eqref{BV_estimate-d} readily follows.
\end{proof}

As we did in Section \ref{ITM} for the \eqref{eq:model} equation. We now prove that the numerical approximation of the solution of Equation \eqref{eq:delay} $n(t,s)$, which is constructed by a simple piece-wise linear interpolation, has a limit when the time step $\Delta t$ and age step $\Delta s$ converge to $0$. 

\begin{lemma}
\label{lim-delta-d}
    Assume that $n^0\in BV(\R^+)$ is compactly supported and the coefficient $p$ satisfies the hypothesis of Theorem \ref{wellp-delay}. Consider the function $n_{\Delta t,\Delta s}(t,s)\in\mathcal{C}\left([0,T],L^1(\R^+)\right)$ defined by
    \begin{equation*}
    n_{\Delta t,\Delta s}(t,s)\coloneqq\frac{t^m-t}{\Delta t}\sum_{j\in\N}n_j^{m-1}\chi_{[s_{j-\frac{1}{2}},s_{j+\frac{1}{2}}]}(s)+\frac{t-t^{m-1}}{\Delta t}\sum_{j\in\N}n_j^{m}\chi_{[s_{j-\frac{1}{2}},s_{j+\frac{1}{2}}]}(s)\;\textrm{if}\;t\in[t^{m-1},t^{m}].
    \end{equation*}
    Then there exists a sub-sequence $(\Delta t_k,\Delta s_k)\to (0,0)$ when $k\to\infty$ and a function $\overline{n}(t,s)$ such that $n_{\Delta t_k,\Delta s_k}\to\overline{n}$ in $\mathcal{C}\left([0,T],L^1(\R^+)\right)$. Moreover, if we define the function $X_{\Delta t,\Delta s}(t)\in\mathcal{C}[0,T]$ as the unique solution of the integral equation
    \begin{equation*}
    X(t)=\int_0^t\int_0^\infty\alpha(t-\tau)p(s,X(\tau))n_{\Delta t,\Delta s}(\tau,s)\,\dd s\,\dd \tau,
    \end{equation*}
   then there exists $\overline{X}\in\mathcal{C}[0,T]$ such that $X_{\Delta t_k,\Delta s_k}\to\overline{X}$ in $\mathcal{C}[0,T]$ and $\overline{X}$ is a solution of the equation
   \begin{equation}
   \label{lim-Xbar}
       \overline{X}(t)=\int_0^t\int_0^\infty\alpha(t-\tau)p(s,\overline{X}(\tau))\overline{n}(\tau,s)\,\dd s\,\dd \tau,
   \end{equation}
   and similarly $N_{\Delta t,\Delta s}\in\mathcal{C}[0,T]$ defined as 
  \begin{equation*}
   N_{\Delta t,\Delta s}(t)=\int_0^\infty p\left(s,X_{\Delta t,\Delta s}(t)\right)n_{\Delta t,\Delta s}\,\dd s,
  \end{equation*}
   converges respectively to $\overline{N}\in\mathcal{C}[0,T]$, where $\overline{N}$ satisfies the equality
   \begin{equation}
    \label{lim-Nbar-d}
       \overline{N}(t)=\int_0^\infty p(s,\overline{X}(t))\overline{n}(t,s)\,\dd s.
   \end{equation}
 \end{lemma}

\begin{proof}
    The proof on the compactness of $n_{\Delta t,\Delta s}$ is the same as Lemma \ref{lim-delta} by using Lemma \ref{BV-lem-d}. Hence there exists a sub-sequence $(\Delta t_k,\Delta s_k)\to (0,0)$ when $k\to\infty$ and a function $\overline{n}(t,s)$ such that $n_{\Delta t_k,\Delta s_k}\to\overline{n}$ in $\mathcal{C}\left([0,T],L^1(\R^+)\right)$.  For the sequence $X_{\Delta t,\Delta s}$ observe that
    \begin{equation*}
    \|X_{\Delta t,\Delta s}\|_\infty\le \|\alpha\|_1\|p\|_\infty\|n^0\|_1,
    \end{equation*}
    and for the derivative we get
    \begin{equation*}
        \begin{split}
           \frac{d}{dt}X_{\Delta t,\Delta s}(t)&=\int_0^\infty\alpha(0)p(s,X_{\Delta t,\Delta s}(t))n_{\Delta t,\Delta s}(\tau,s)\,\dd s\,\dd \tau+\int_0^t\int_0^\infty\alpha'(t-\tau)p(s,X_{\Delta t,\Delta s}(\tau))n_{\Delta t,\Delta s}(\tau,s)\,\dd s\,\dd \tau,
        \end{split}
    \end{equation*}
    thus we have the following estimate
 \begin{equation*}
    \left\|\frac{d}{dt}X_{\Delta t,\Delta s}\right\|_\infty\le \alpha(0)\|p\|_\infty\|n^0\|_1+T\|\alpha'\|_\infty\|p\|_\infty\|n^0\|_1,    
  \end{equation*}
    and from Arzelà-Ascoli theorem we conclude that $X_{\Delta t,\Delta s}$ converges to some $\overline{X}$ in $\mathcal{C}[0,T]$ for a sub-sequence. By passing to the limit, Equations \eqref{lim-Xbar} and \eqref{lim-Nbar-d} readily follow.
\end{proof}

We thus obtain the convergence result for the numerical scheme corresponding to the \eqref{eq:delay} equation.

\begin{theorem}[Convergence of the numerical scheme]
Assume that $n^0\in BV(\R^+)$ is compactly supported and the coefficient $p$ satisfies the hypothesis of Theorem \ref{wellp-delay}. Then for all $T>0$, the numerical scheme converges to the unique weak solution $n\in\mathcal{C}([0,T],L^1(\R^+))$ of the equation \eqref{eq:delay}.
\end{theorem}

\begin{proof}
    The proof is the same as Theorem \ref{conv-num}.
\end{proof}

\begin{remark}
The previous results are also valid for the case when the coefficient $p$ is of the form 
\begin{equation*}
p(s,X)=\varphi(X)\chi_{\{s>\sigma(X)\}},
\end{equation*}
with $\varphi$ and $\sigma$ Lipschitz bounded functions.
\end{remark}

\section{Numerical Results}
\label{Simu}
In order to illustrate the theoretical results of the previous sections, we present in different scenarios for the dynamics of the equations  \eqref{eq:model} and \eqref{eq:delay}, which are solved by the finite volume method described in Algorithms \eqref{Alg} and \eqref{Alg2} respectively, while the nonlinear problems \eqref{eqfijo-N} and \eqref{eqfijo-X} where solved for $N$ or $X$ using the Newton-Raphson iterative method with a relative error less than $10^{-12}$. 
Additionally, we use Matlab software for all our simulations.
For all numerical tests, we consider the prototypical coefficient $p$ with absolute refractory period $\sigma>0$ given by
\begin{equation*}
p(s,N)=\varphi(N)\chi_{\{s>\sigma\}}(s),  
\end{equation*}
where $\varphi(N)$ and $\sigma$ are specified in each example. 
For the \eqref{eq:delay} equation we will consider for the delay kernel $\alpha(t)$ the following examples
\begin{equation*}
 \alpha_1(t)=\frac{e^{-t/\lambda }}{\lambda} \quad \text{or} \quad \alpha_2(t)=\frac{1}{\sqrt{2\pi}\lambda}e^{-\frac12(\frac{t-d}{\lambda})^2},\quad \text{with } \lambda=10^{-3}.     
\end{equation*}

For this choice of $\lambda$ we essentially consider the approximation $\alpha_1(t)\approx \delta(t)$, where we are interested in comparing both \eqref{eq:model} and \eqref{eq:delay} equations when we are close to this limit case. Similarly for the second kernel we get $\alpha_2(t)\approx \delta(t-d)$ and we study the behavior of \eqref{eq:delay} equation with different values of the parameter $d$.  

\subsection*{Example 1: A strongly inhibitory case}
We start with an inhibitory firing coefficient, i.e. $\varphi'(N)<0$, given by
\begin{equation}\label{exp:inhibitory}
\varphi(N)=e^{-9N},\quad \sigma=\frac12,  \quad n^0(s)=\frac12e^{-(s-1)^+}.
\end{equation}
With these choice of parameters \eqref{eq:model} equation has a unique steady state with $N^*\approx 0.1800$, by solving Equation \eqref{eqfix-N}. For the \eqref{eq:model} equation we display in Figures~\ref{fig:inhib-nodelay}(a-b) the numerical solution for $n(t,s)$ with $(t,s)\in[0,30]\times[0,40]$ and $N(t)$ with $t\in[0,10]$, where we observe that the total activity $N(t)$ converges to the unique steady state $N^*$, while for $n(t,s)$ the initial condition moves to the right (with respect to age $s$) as it decays exponentially and approaches the equilibrium density given by Equation \ref{solest1}. 
Moreover, this example clearly aligns with the theoretical results on convergence to equilibrium in the strongly excitatory case
studied in \cite{torres2021elapsed}.

\begin{figure}[t]
\begin{tabular}{cc}
  (a)& (b)  \\
  \includegraphics[width=0.45\textwidth]{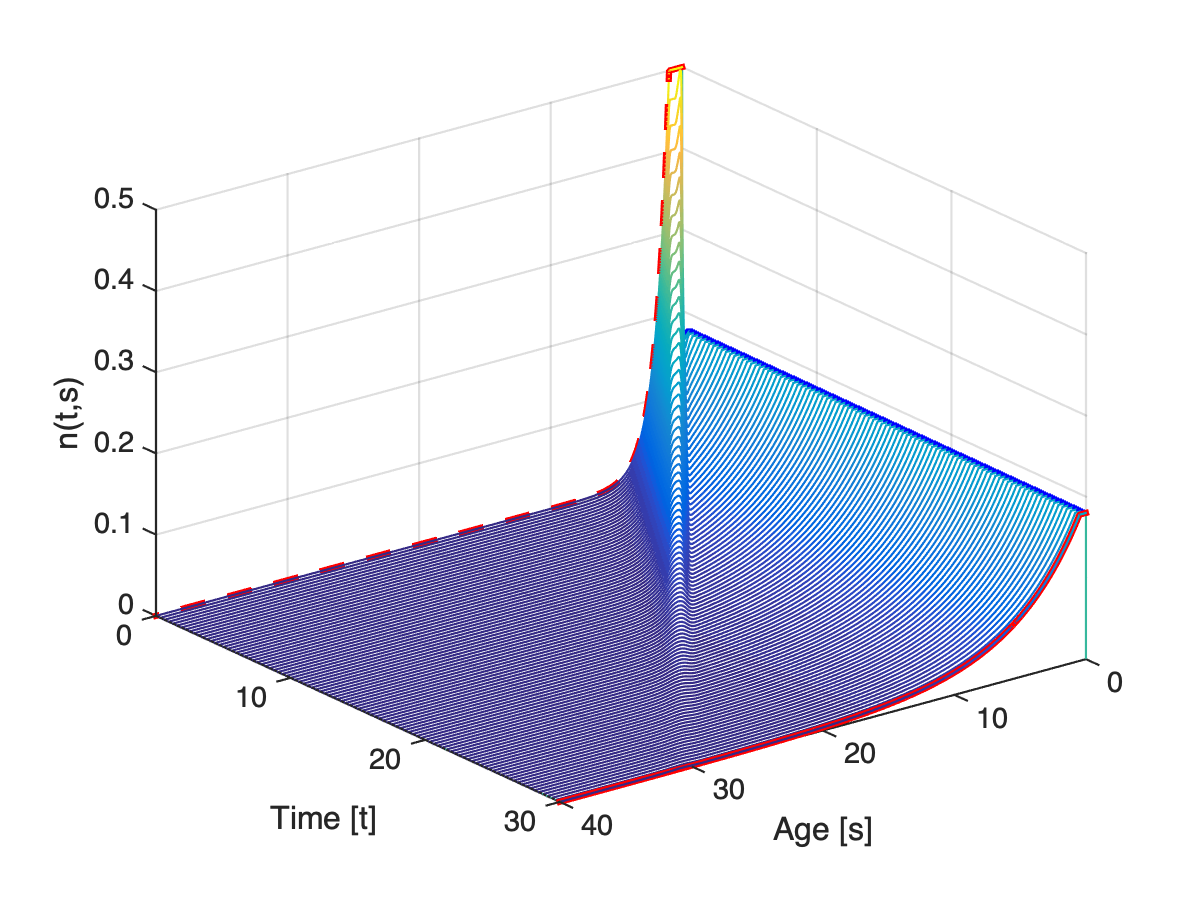}  & \includegraphics[width=0.45\textwidth]{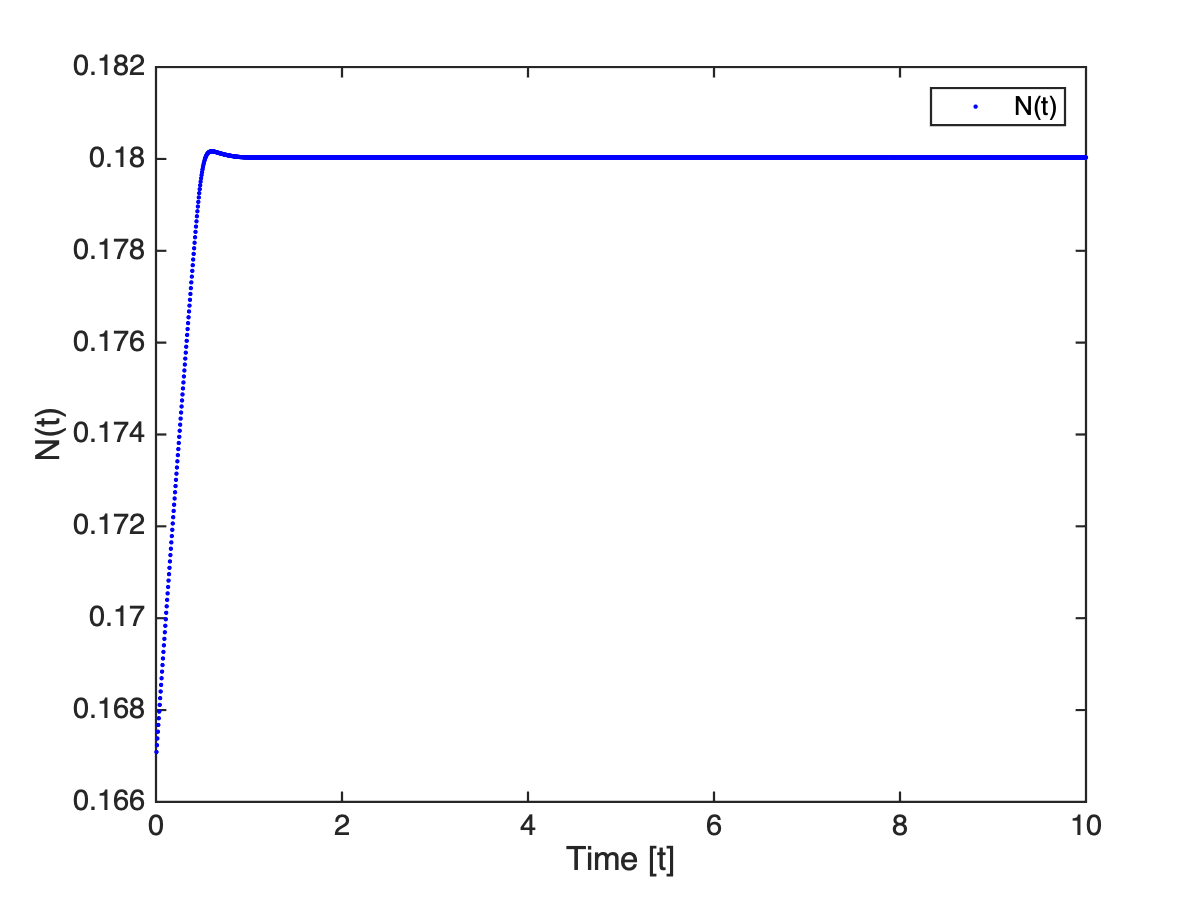} \\
    (c)& (d)  \\
  \includegraphics[width=0.45\textwidth]{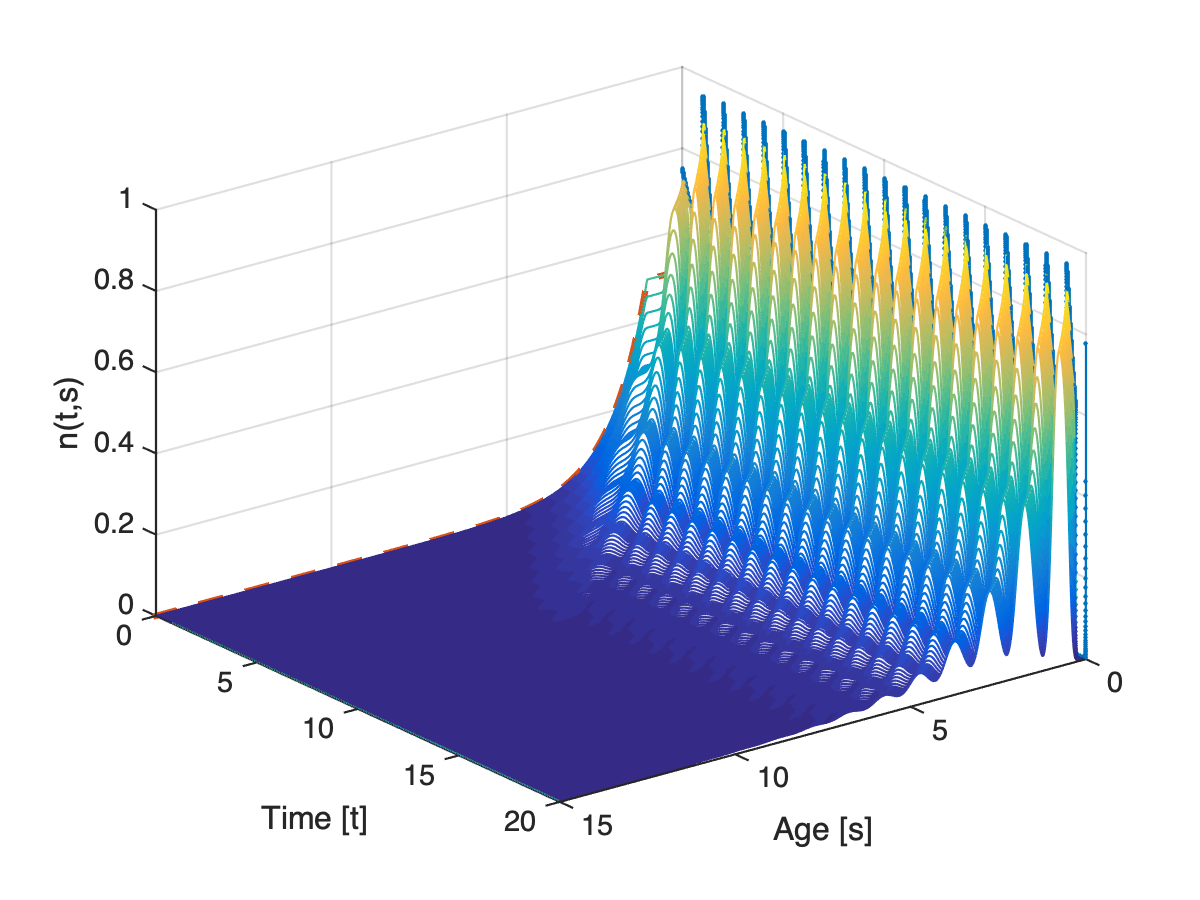}  & \includegraphics[width=0.45\textwidth]{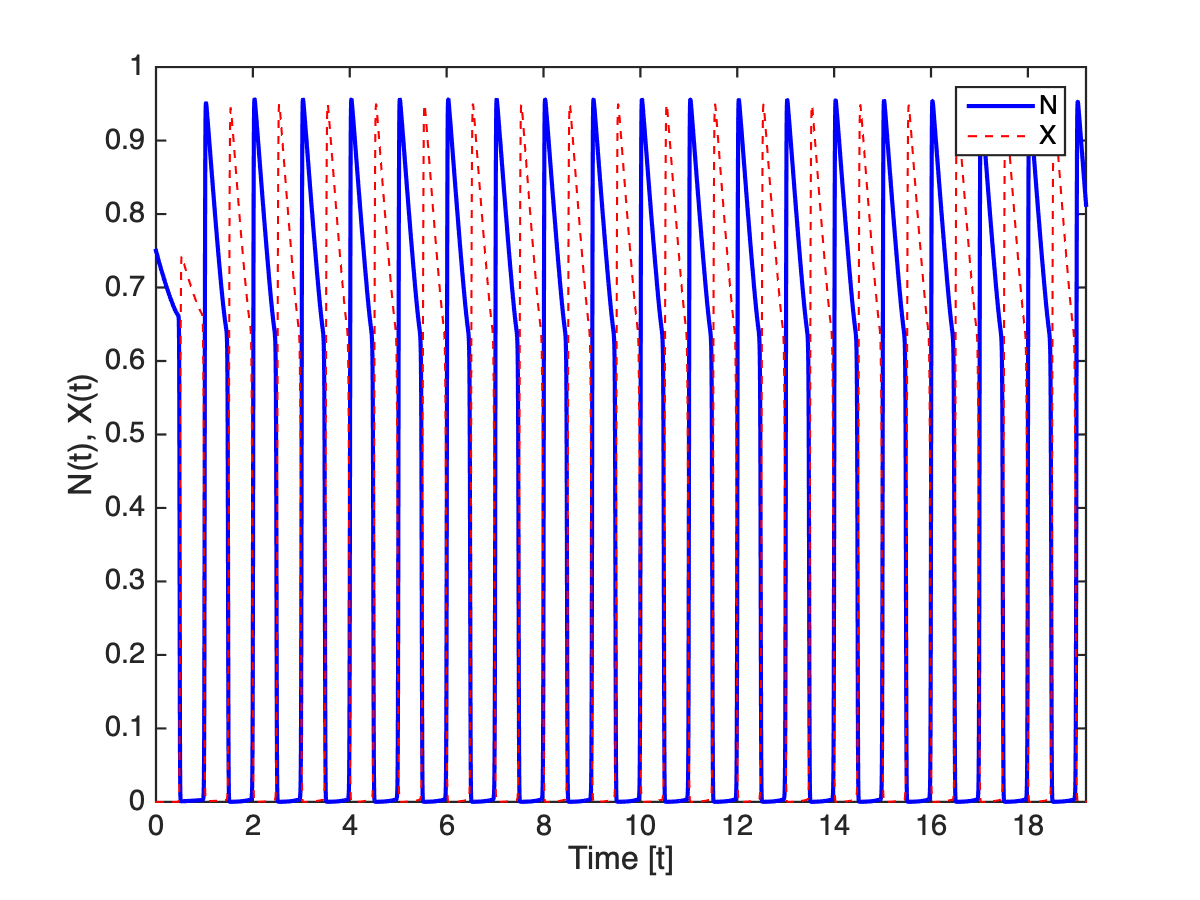} \\
\end{tabular}
    \caption{An inhibitory case. (a-b) Density $n(t,s)$ and discharging flux $N(t)$ for \eqref{eq:model}. (c-d) Density $n(t,s)$, discharging flux $N(t)$ and total activity $X(t)$ for \eqref{eq:delay} with $\alpha_2(t)\approx \delta(t-\frac12)$.}
    \label{fig:inhib-nodelay}
\end{figure}

On the other hand, we solve the \eqref{eq:delay} with the parameters in \eqref{exp:inhibitory} and we choose $d=\frac12$ so that $\alpha_2(t)\approx \delta(t-\frac12)$. In Figs~\ref{fig:inhib-nodelay}(c,d) we display numerical solution for $n(t,s)$ for $(t,s)\in[0,15]\times[0,20]$, and $N(t)$ and $X(t)$ for $t\in[0,20]$. Unlike \eqref{eq:model} equation, the solutions for $N$ and $X$ in \eqref{eq:delay} converge to a periodic profile that we conjecture to be $2d$-periodic and they tend to differ by a period of time equal to $d$, which means that $|N(t-d)-X(t)|\approx 0$ for $t$ large. The density $n(t,s)$ is asymptotic to its respective periodic profile in time. In this case, we observe a periodic solution induced by a negative feedback delay, which is a classical behavior in the context of delay differential equations. This negative feedback corresponds to the inhibition determined by the coefficient $p$, so that when combined with the delay it induces cycles of increase and decrease in the discharging flux $N$ and total activity $X$, which are favorable to form periodic solutions (see \cite{diekmann2012delay} for a reference).

\subsection*{Example 2: An excitatory case with a unique steady state} 
Now we consider an excitatory case, i.e. $\varphi'(N)>0$, given by 
\begin{equation}\label{exp2:datums}
  \varphi(N)=\frac{10N^2}{N^2+1}+0.5,\quad \sigma=1,  \quad n^0(s)=e^{-(s-1)^+}\chi_{\{s>1\}}(s).
\end{equation}
This example was previously studied in \cite{torres2021elapsed} for the \eqref{eq:model} equation. We know that under this choice of parameters the system has a unique steady state with $N^*\approx 0.8186$. 

For the \eqref{eq:model} equation, in Figure \ref{Exp2}(a) we display the numerical solution for $n(t,s)$ for $(t,s)\in[0,20]\times[0,4]$ and in  the blue curve of Figure \ref{Exp2}(b) we show $N(t)$ for $t\in[0,20]$ where the solution is asymptotic periodic pattern with jump discontinuities. This is due to the invertibility condition $\Psi(N,n)$ in \eqref{cond:inverti} is close to zero as we show in Figure \ref{Exp2}(c), where we plot $\Psi(N(t),n(t,\cdot))$ for $t\in[0,10]$. We observe that when a discontinuity arises for $N(t)$ in Figure \ref{Exp2}(b), then $\Psi(N,n)$ is close to zero in Figure \ref{Exp2}(c). This means that the invertibility condition \eqref{cond:inverti} is a key criterion that determine the existence and continuity of solutions.

For the \eqref{eq:delay} equation with $\alpha_1(t)\approx\delta(t)$ we observe in the red curve of Figure \ref{Exp2}(b) that the respective discharging flux is a smooth approximation of the activity of $N(t)$ for the \eqref{eq:model} equation. This is due to the regularizing effect of the delay kernel $\alpha$ through the convolution. 

\begin{figure}[ht!]
 \begin{tabular}{cc}
     (a) & (b)  \\
    \includegraphics[width=0.45\textwidth]{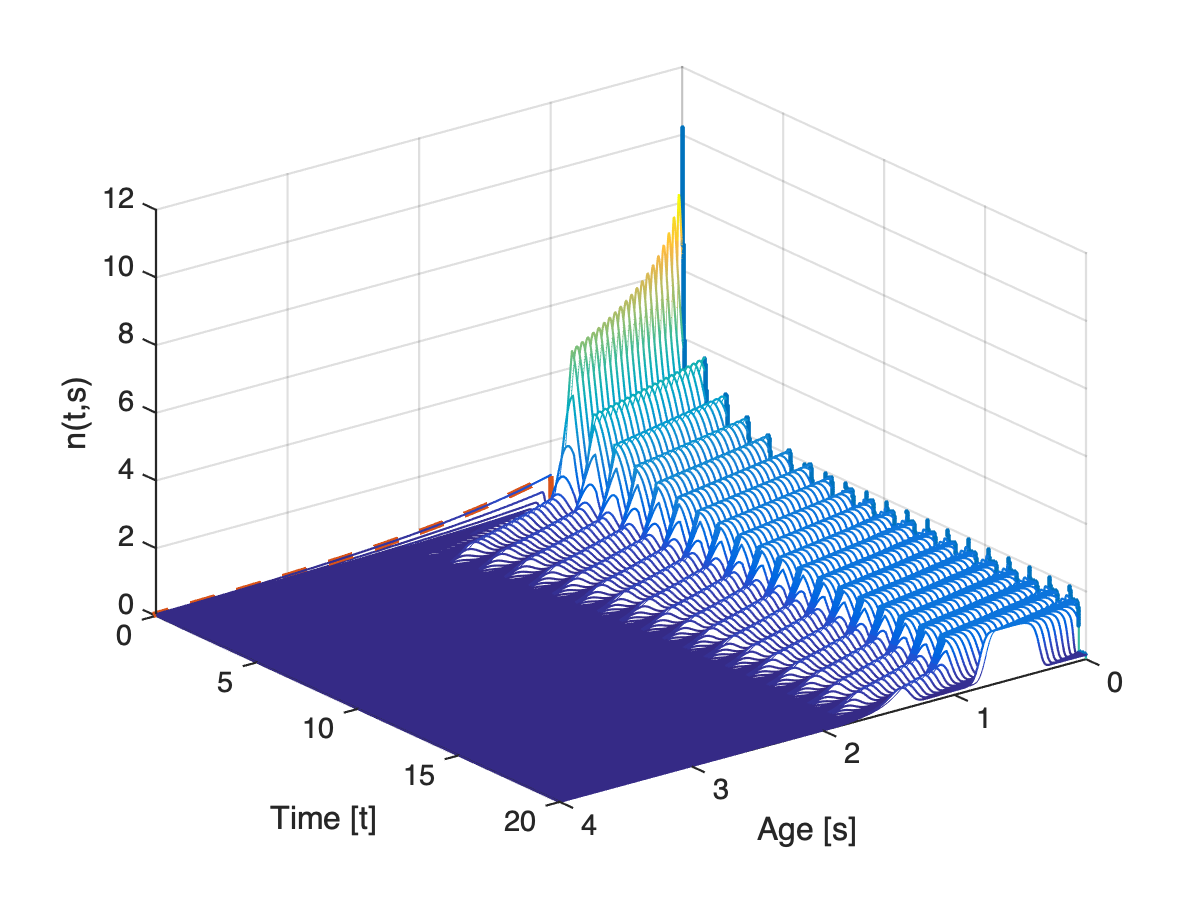}  & 
    \includegraphics[width=0.45\textwidth]{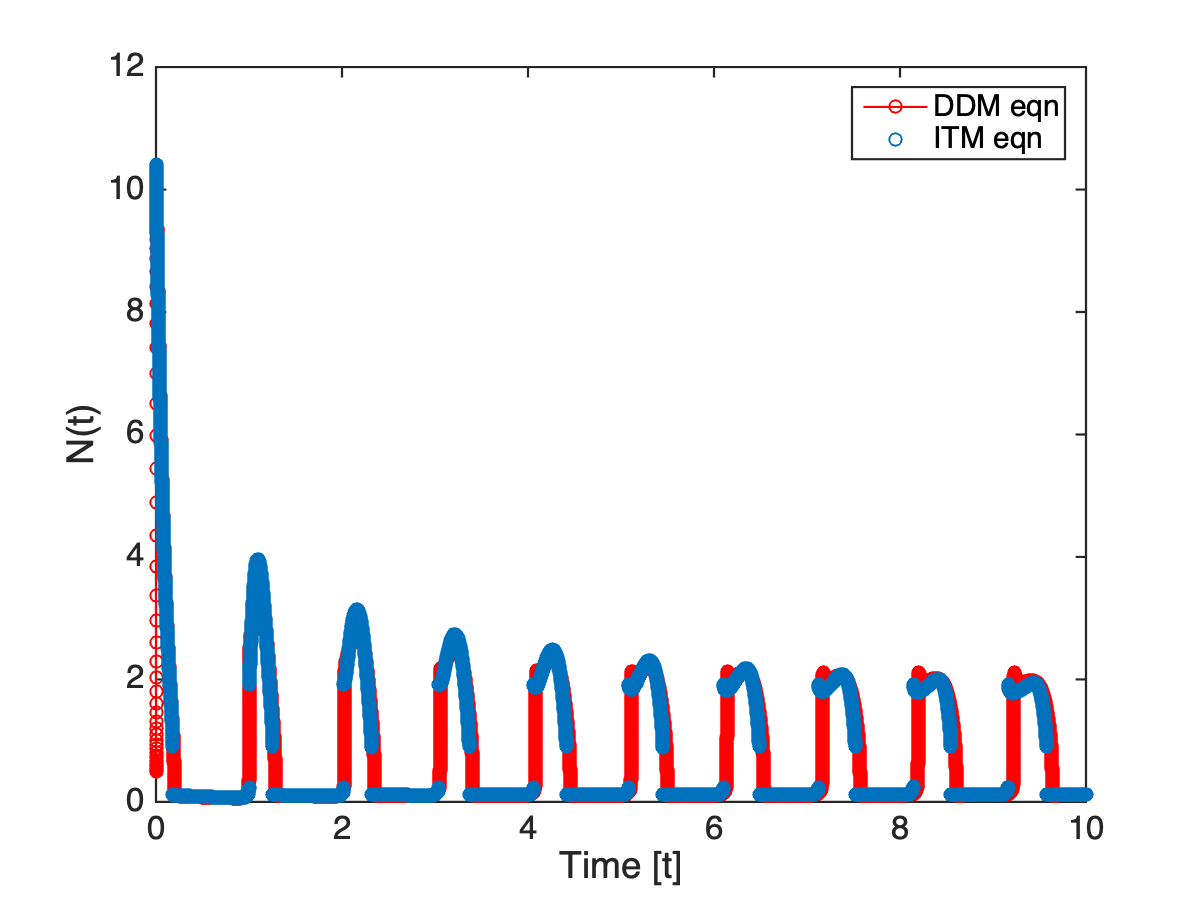}\\  
  \multicolumn{2}{c}{(c)}             \\ 
 \multicolumn{2}{c}{ \includegraphics[width=0.45\textwidth]{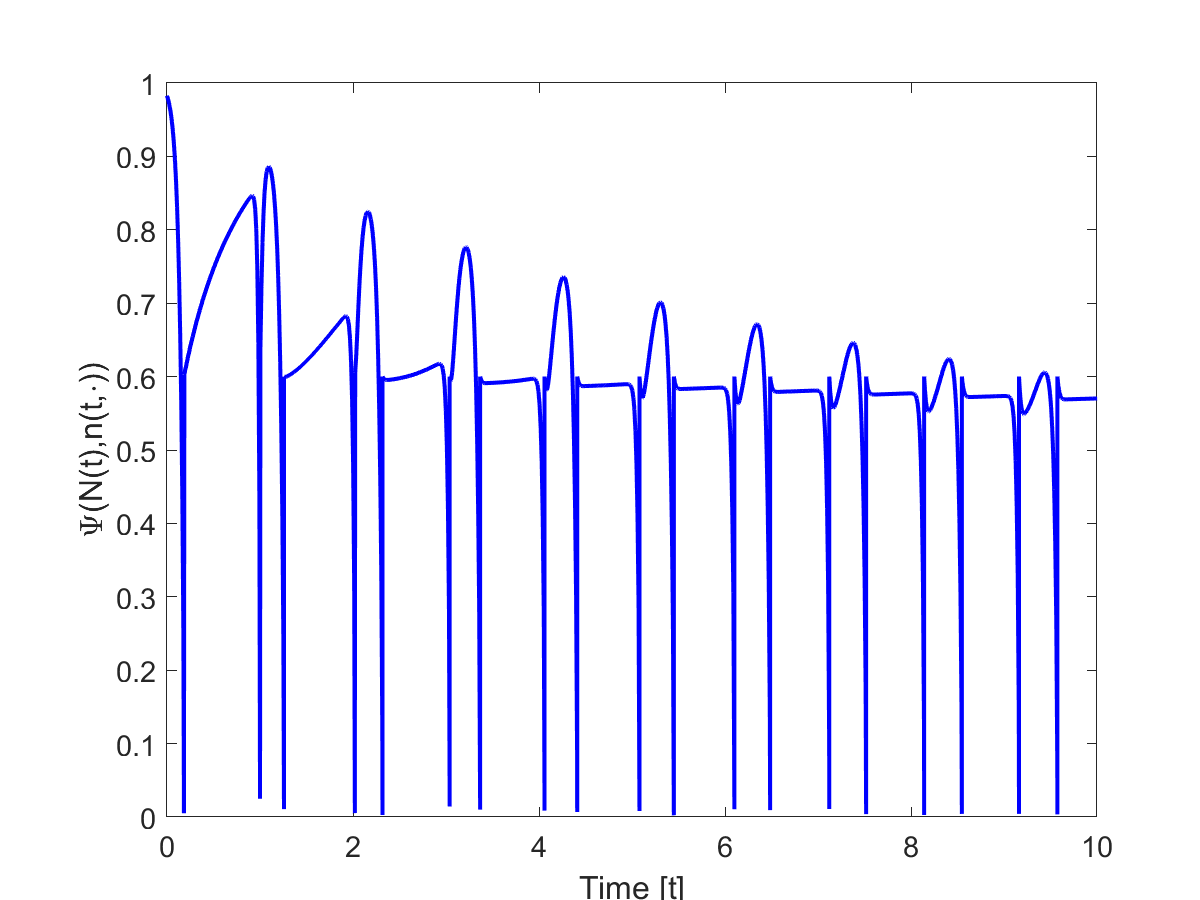}
 }
\end{tabular}
    \caption{An excitatory case with periodic patterns. (a) Density $n(t,s)$ for \eqref{eq:model}, (b) Comparison of $N(t)$ for both equations \eqref{eq:model} and \eqref{eq:delay} with $\alpha_1(t)\approx\delta(t)$, (c) Invertibility condition $\Psi(N,n)$ for \eqref{eq:model}.}
    \label{Exp2}
\end{figure}

Next we take $d=1$ so that $\alpha_2(t) \approx \delta(t-1)$. In Figure \ref{Exp2b}(a) we display $n(t,s)$ for $(t,s)\in[0,6]\times[0,15]$ and in Figure \ref{Exp2b}(b) we display the graphics of $N(t)$ and $X(t)$ for $t\in[0,15]$, where we observe an asymptotic periodic pattern for both discharging flux $N$ and total activity $X$, which we conjecture to be $d$-periodic. In this case we observe a a synchronization phenomena which means that $|N(t)-X(t)|\approx 0$ for large $t$, unlike the inhibitory case shown in Fig \ref{fig:inhib-nodelay}(d) where they tend to differ in time by $d$. In this excitatory case we conjecture that periodic solutions are due to effect of the refractory period $\sigma$ as it was studied in \cite{torres2021elapsed} and the solutions are continuous due to the regularizing effect of the kernel $\alpha$. Therefore we observe that periodic solutions that may arise in the inhibitory and excitatory are of different nature. Indeed, similar phenomena has been previously observed in related models such as the Fokker-Planck equation \cite{caceres2018analysis,caceres2019global,cormier2021hopf,ikeda2022theoretical}. In the excitatory case, periodic solutions have been observed under a refractory regime; while in absence of this period, there are no periodic solutions in the strongly excitatory, regardless the delay. On the other hand for the inhibitory case, periodic solutions have been observed with large delay but regardless the refractory period, suggesting the different nature of periodic solutions.

\begin{figure}[ht!]
   \begin{tabular}{cc}
       (a)& (b) \\
\includegraphics[width=0.45\textwidth]{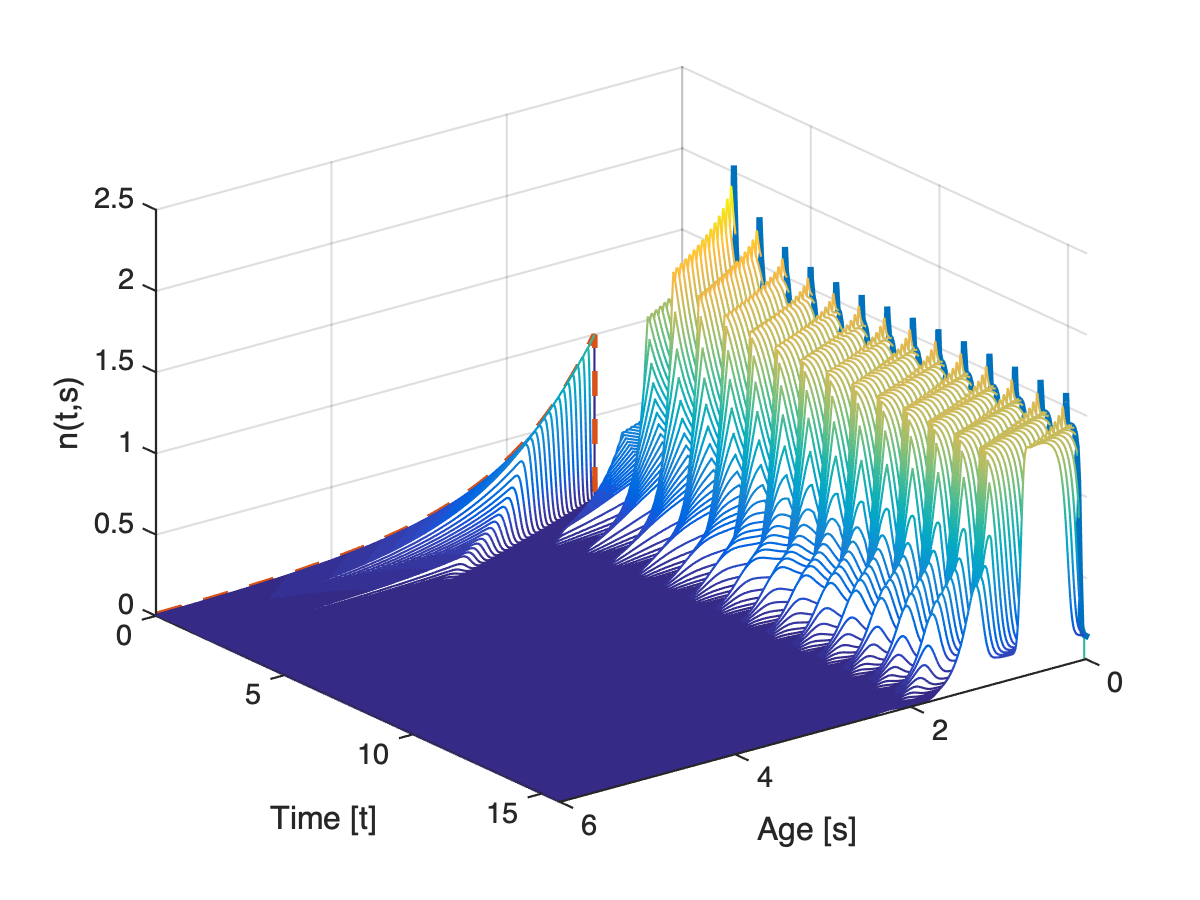} &
\includegraphics[width=0.45\textwidth]{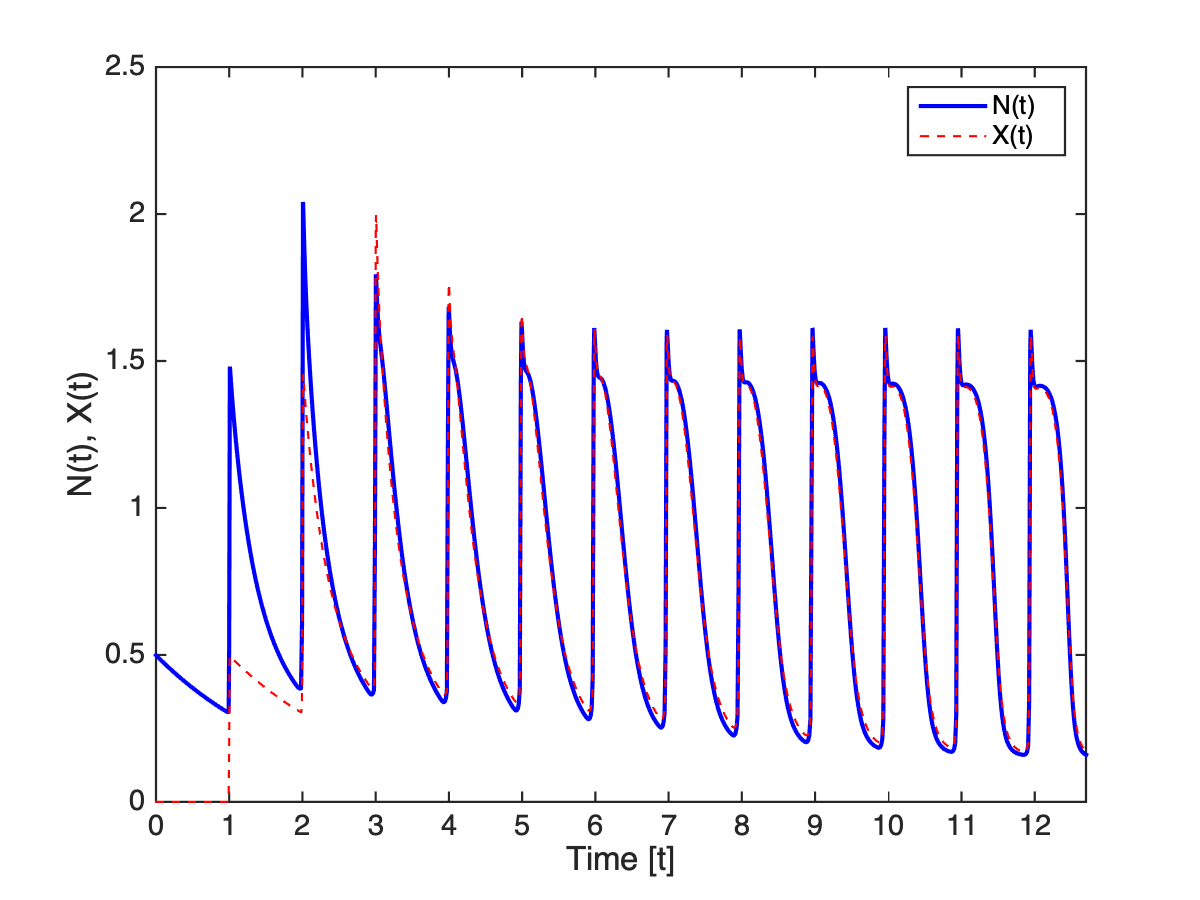}\\
\end{tabular}
    \caption{A periodic solution for the \eqref{eq:delay} equation. (a-b) Density $n(t,s)$, discharging flux $N(t)$ and total activity $X(t)$ with $d=1$ and $\alpha_2(t)\approx \delta(t-1),\,X(t)\approx N(t-1)$.}
    \label{Exp2b}
\end{figure}

\subsection*{Example 3: An excitatory case with multiples steady states} 
Next, we consider a case where $\varphi'(N)>0$ with parameters given by
\begin{equation}
   \varphi(N)=\frac{1}{1+e^{-9N+3.5}},\quad \sigma=\frac12, \quad n^0(s)=e^{-(s-\frac12)^+}\chi_{\{s>\frac12\}}(s).
\end{equation}
This example was previously studied in \cite{torres2021elapsed} for the \eqref{eq:model} equation. Under this choice of parameters, we have three different solutions for $N(0)$ according to Equation \eqref{eqN0}, that are given by $N^0_1\approx 0.0410$, $N^0_2\approx 0.3650$ and $N^0_3\approx 0.6118$. These values determine three different branches of local continuous solutions. For the \eqref{eq:model} equation we display the numerical solution for $n(t,s)$ and $N(t)$ in Figures ~\ref{fig:exp_excita_r3}. The dynamics for the discharging flux $N$ is determined by the initial condition $N(0)$ and thus it also determines the dynamics for $n$. In this case, we observe three different types of numerical approach for $N(t)$, which converge to two different equilibrium points $N^*_1$ and $N^*_2$.  \\
\begin{figure}
 \begin{tabular}{cc}
     (a) & (b) \\
\includegraphics[width=0.4\textwidth]{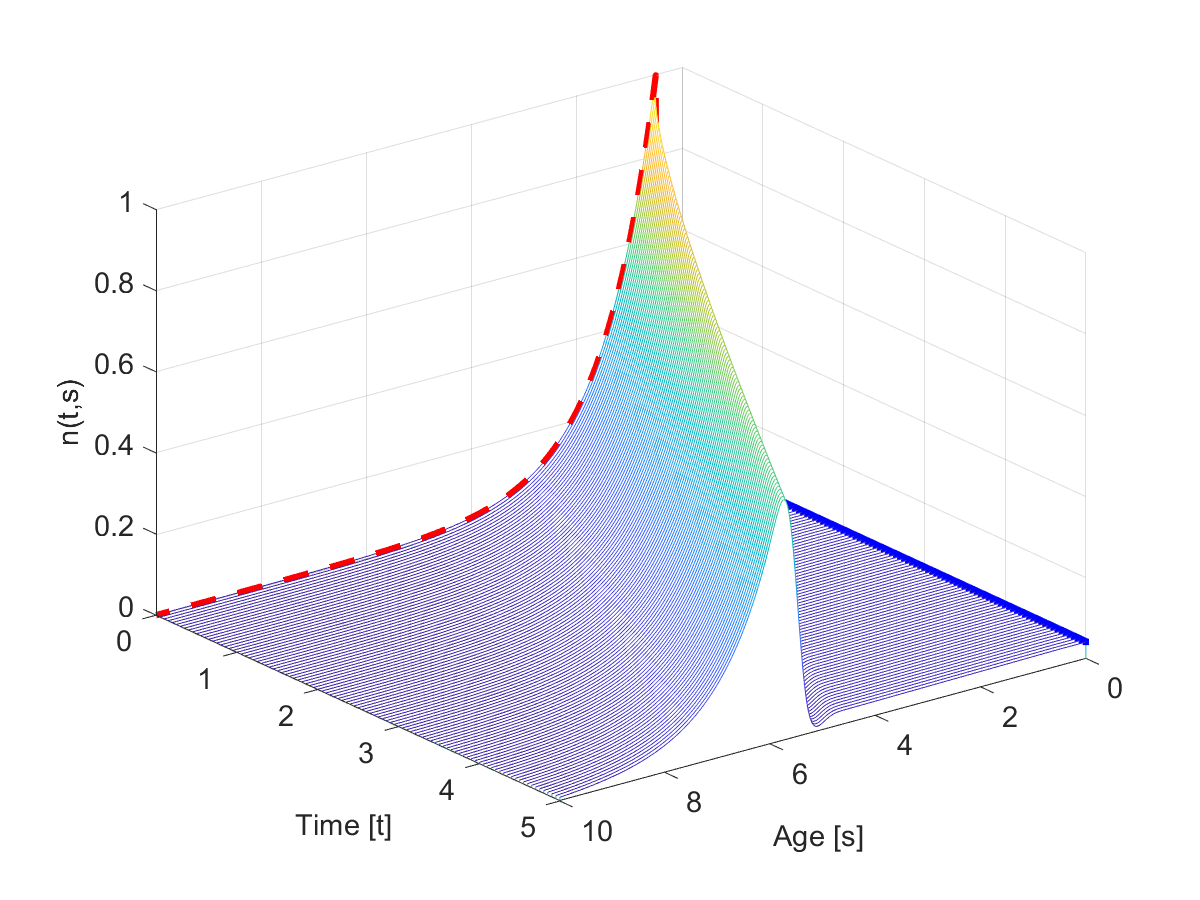}  &  \includegraphics[width=0.4\textwidth]{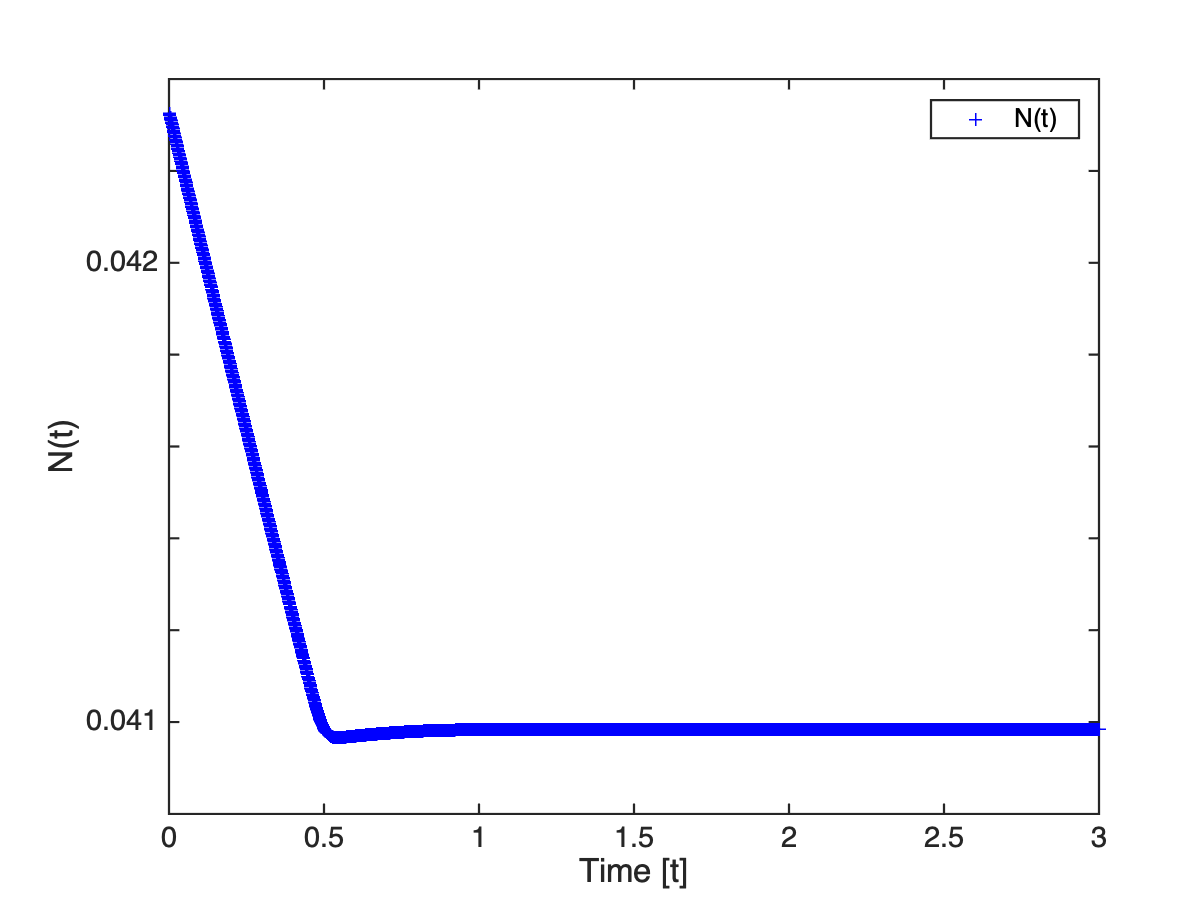} \\
(c) & (d) \\
\includegraphics[width=0.4\textwidth]{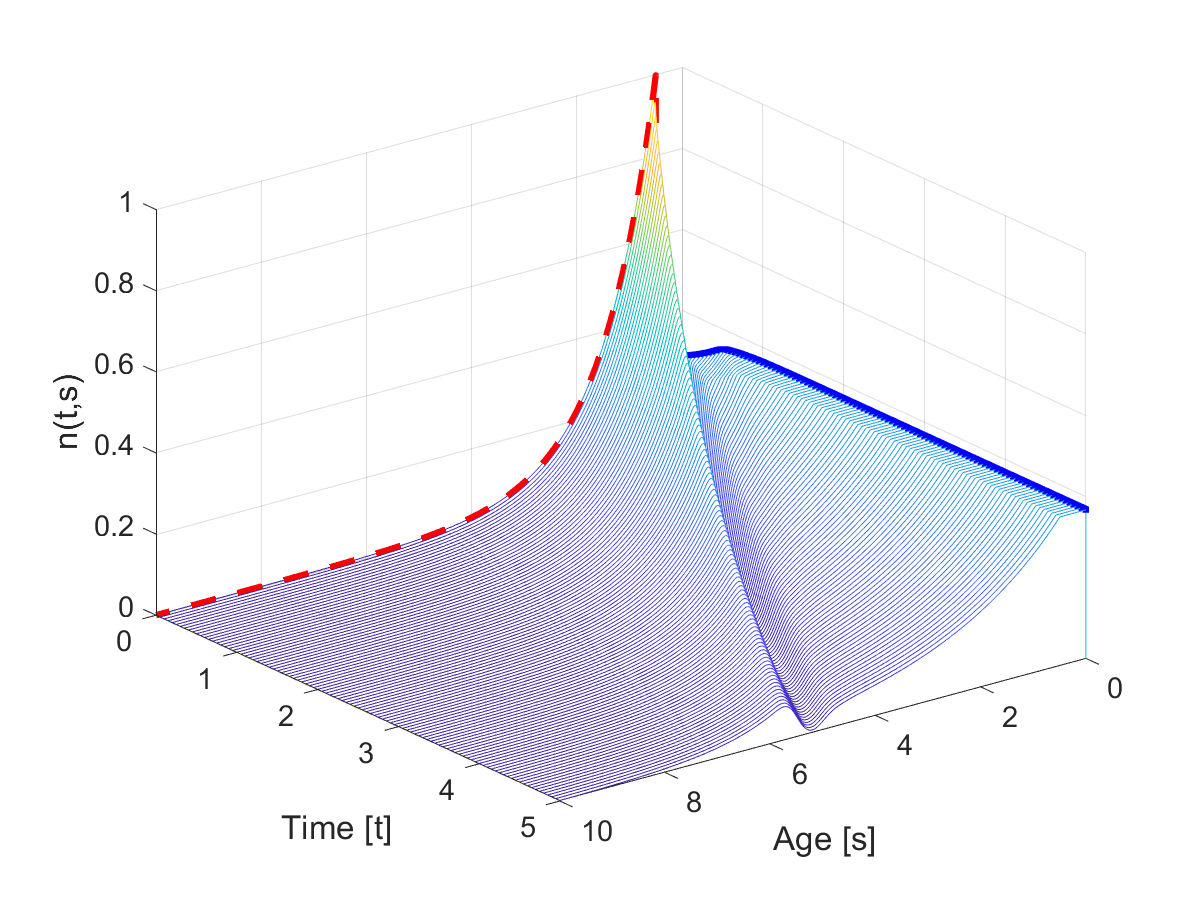} & \includegraphics[width=0.4\textwidth]{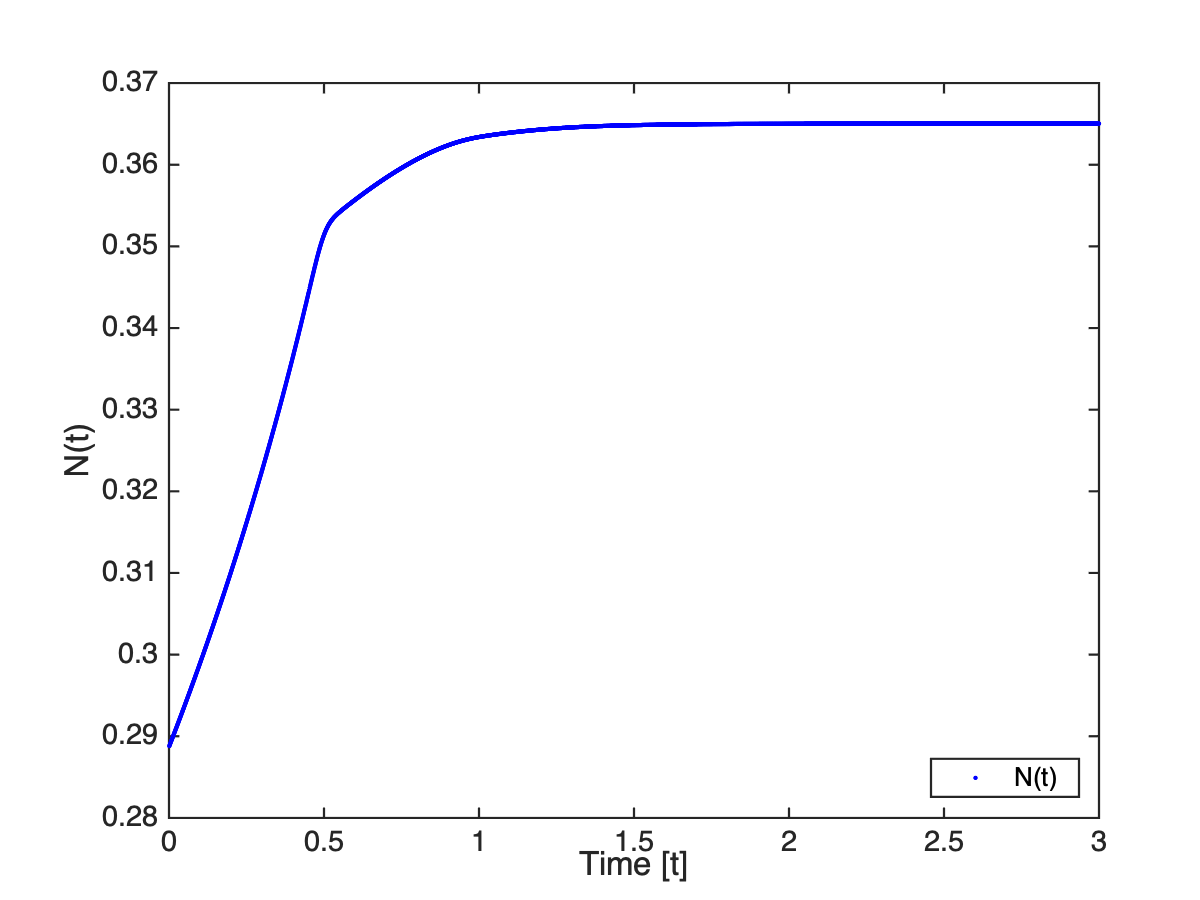} \\
  (e) & (f) \\     
\includegraphics[width=0.4\textwidth]{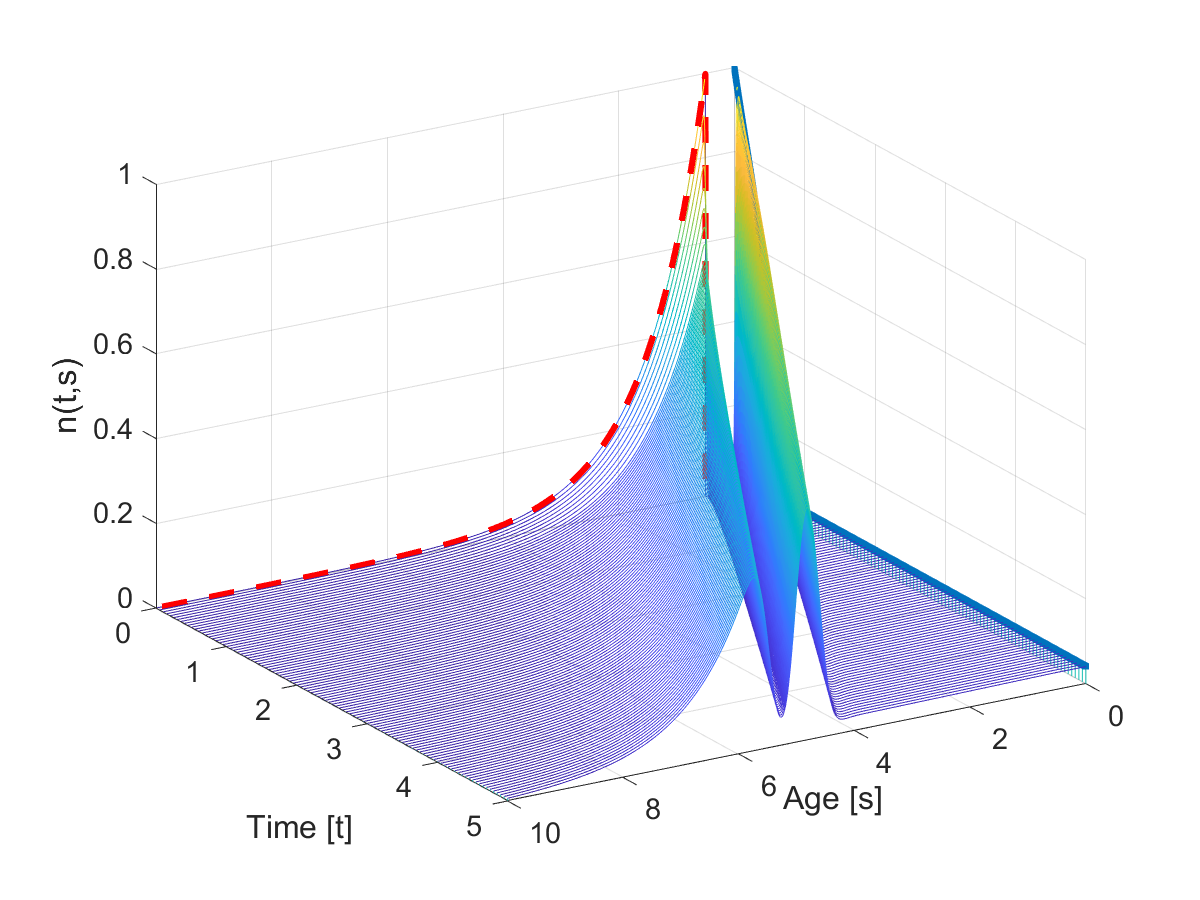} & \includegraphics[width=0.4\textwidth]{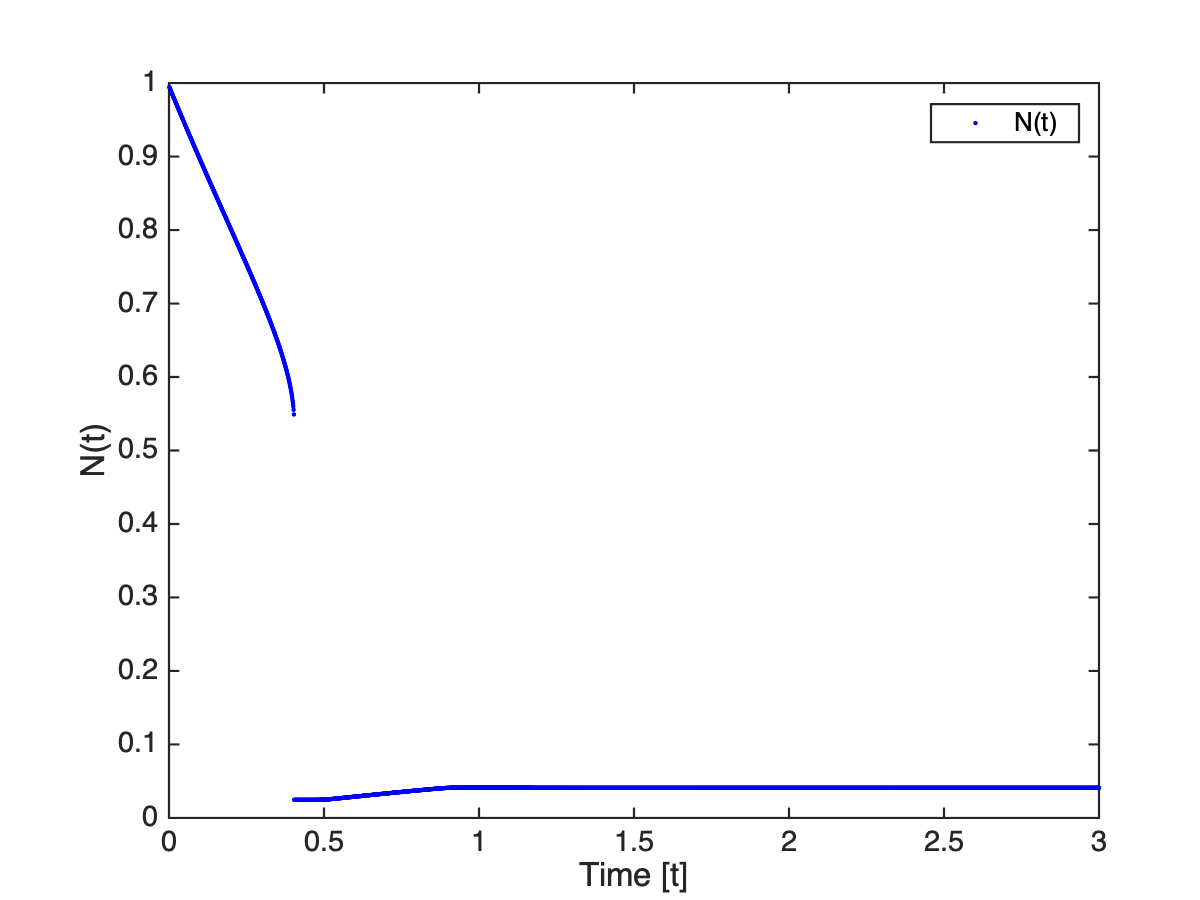}     
 \end{tabular}
    \caption{An excitatory problems with multiples solutions for the  \eqref{eq:model} equation with different initial approximations for $N^0$ in \eqref{eqfijo-N}. (a-b) Density $n(t,s)$ and discharging flux $N(t)$ for $N^0_1\approx 0.0281$, (c-d) for $N^0_2\approx 0.4089$  and (e-f) for $N^0_3\approx 0.7114$}
    \label{fig:exp_excita_r3}
\end{figure}
\\
In the \eqref{eq:delay} equation for both $\alpha_1(t)\approx\delta(t)$ and  $\alpha_2(t)\approx\delta(t-d)$, we observe in Figure \ref{fig:exp_excita_r3_delay} that $N(t)$ converge to the first equilibrium of the \eqref{eq:model} equation. Moreover, when the instantaneous transmission model has multiple branches of solutions for the same initial condition, i.e. multiples solutions for $N(0)$, we conjecture that when $\alpha(t)$ converges $\delta(t)$ in the sense of distributions, the total activity of $X(t)$ in the \eqref{eq:delay} equation converges a.e. to the solution of $N(t)$ in the \eqref{eq:model} equation, whose value of $N(0)$ is the closest one to zero and we expect $L^1$-convergence for the corresponding probability densities. This is due to the condition $X(0)=0$ imposed for the \eqref{eq:delay} equation. We observe in Figure \ref{fig:exp_excita_r3_delay}(b) that $X(t)$ and $N(t)$ follow the same behavior of the Figure \ref{fig:exp_excita_r3}(b) and in particular the total activity $X(t)$ grows fast from $X(0)=0$ until it approaches to the solution of $N(t)$.

Similarly when $\alpha(t)$ converges to $\delta(t-d)$, we conjecture that total activity $X(t)$ in the \eqref{eq:delay} equation converges to the solution of $N(t)$ of Equation \ref{eq:delaypuro} that satisfies $N(t)\equiv 0$ for $t\in [-d,0]$, as it is suggested by the numerical solution in Figure \ref{fig:exp_excita_r3_delay}(d) since we would formally get $X(t)=N(t-d)$ and $X(t)=0$ for $t\in[0,d]$.

\begin{figure}[t]
   \begin{tabular}{cc}
  (a) & (b) \\
    \includegraphics[width=0.45\textwidth]{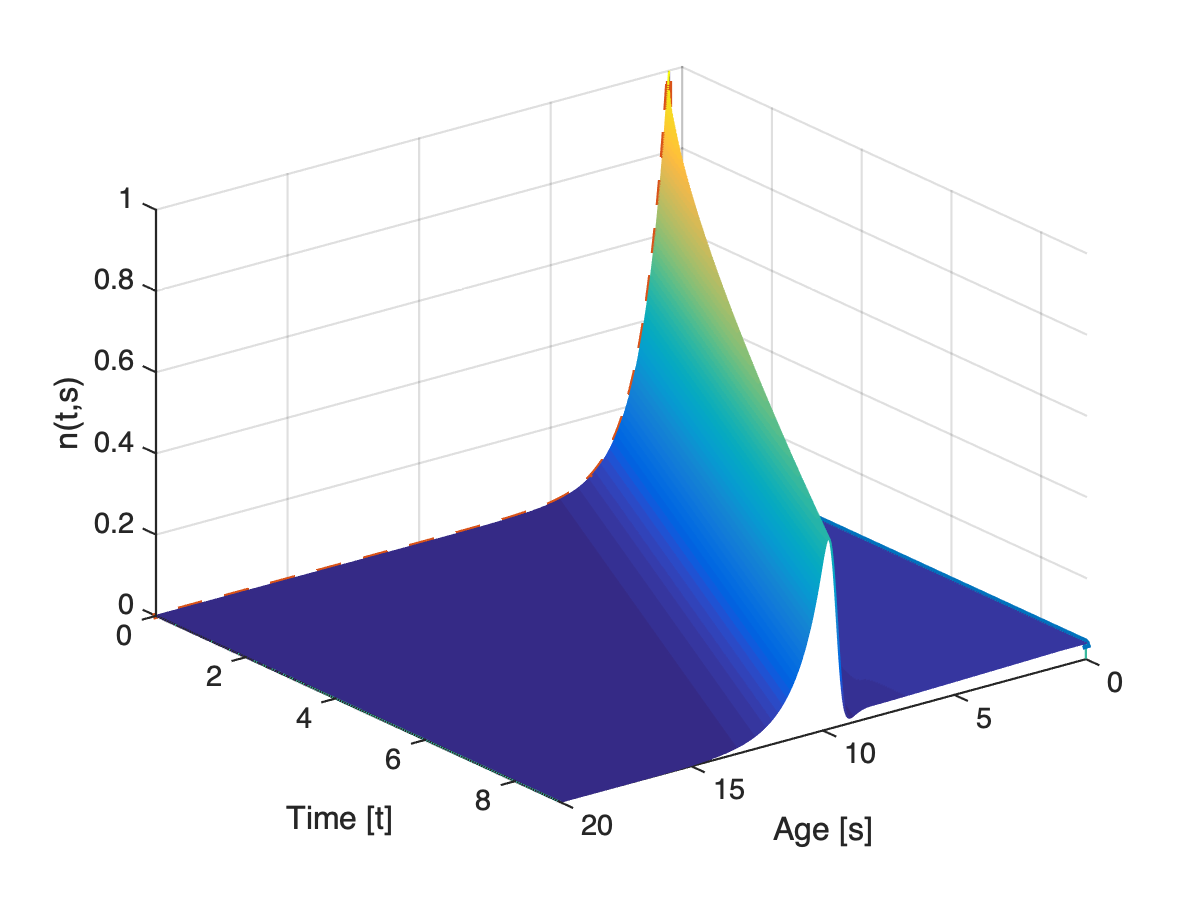}   & 
    \includegraphics[width=0.45\textwidth]{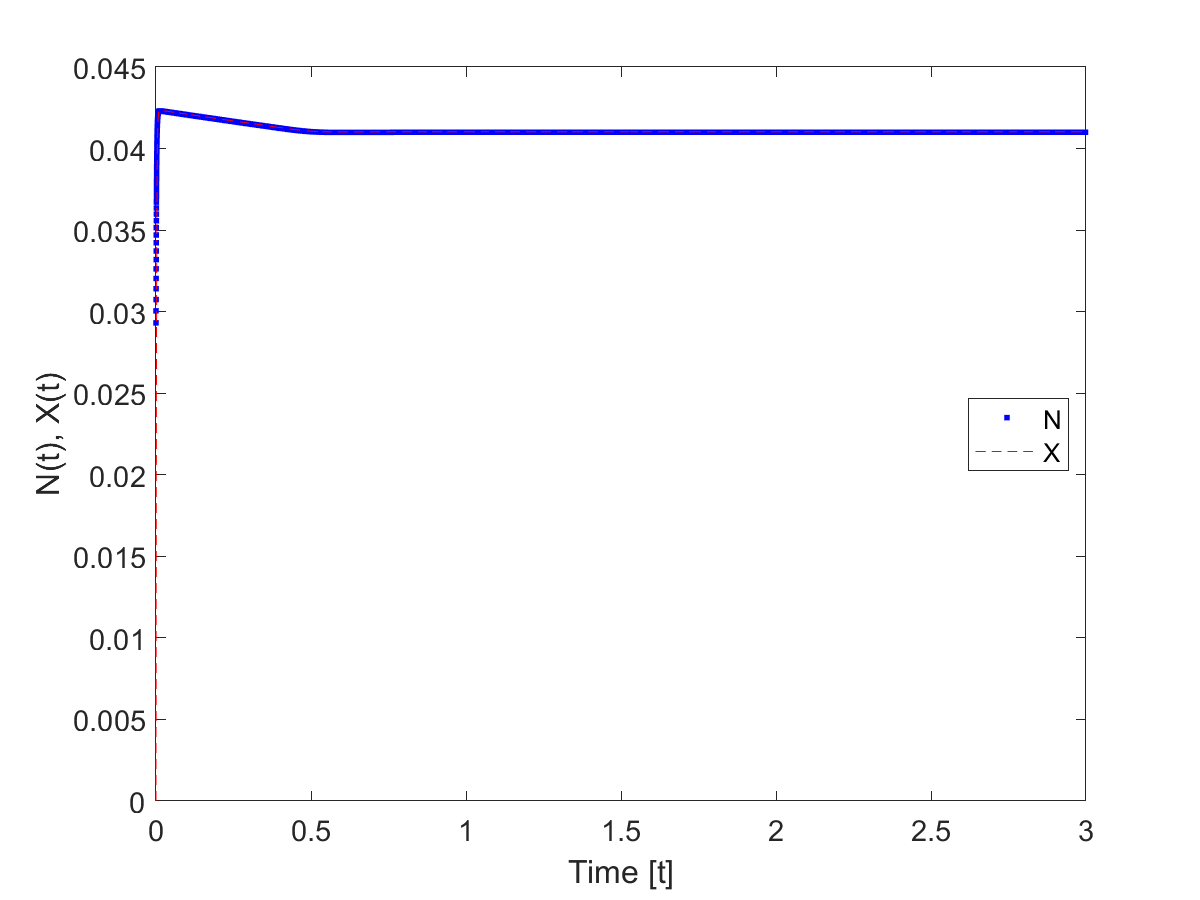} \\
  (c) & (d) \\
    \includegraphics[width=0.45\textwidth]{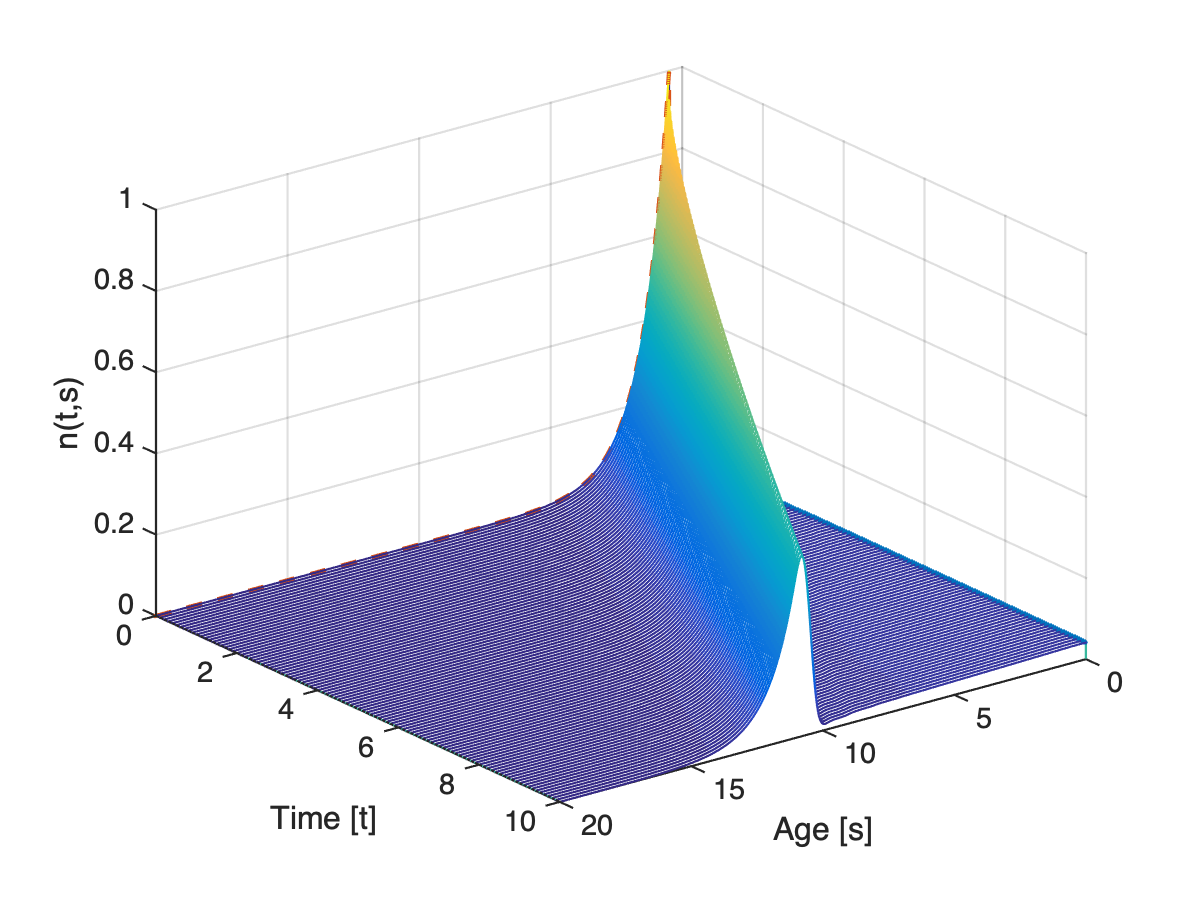}   & 
    \includegraphics[width=0.45\textwidth]{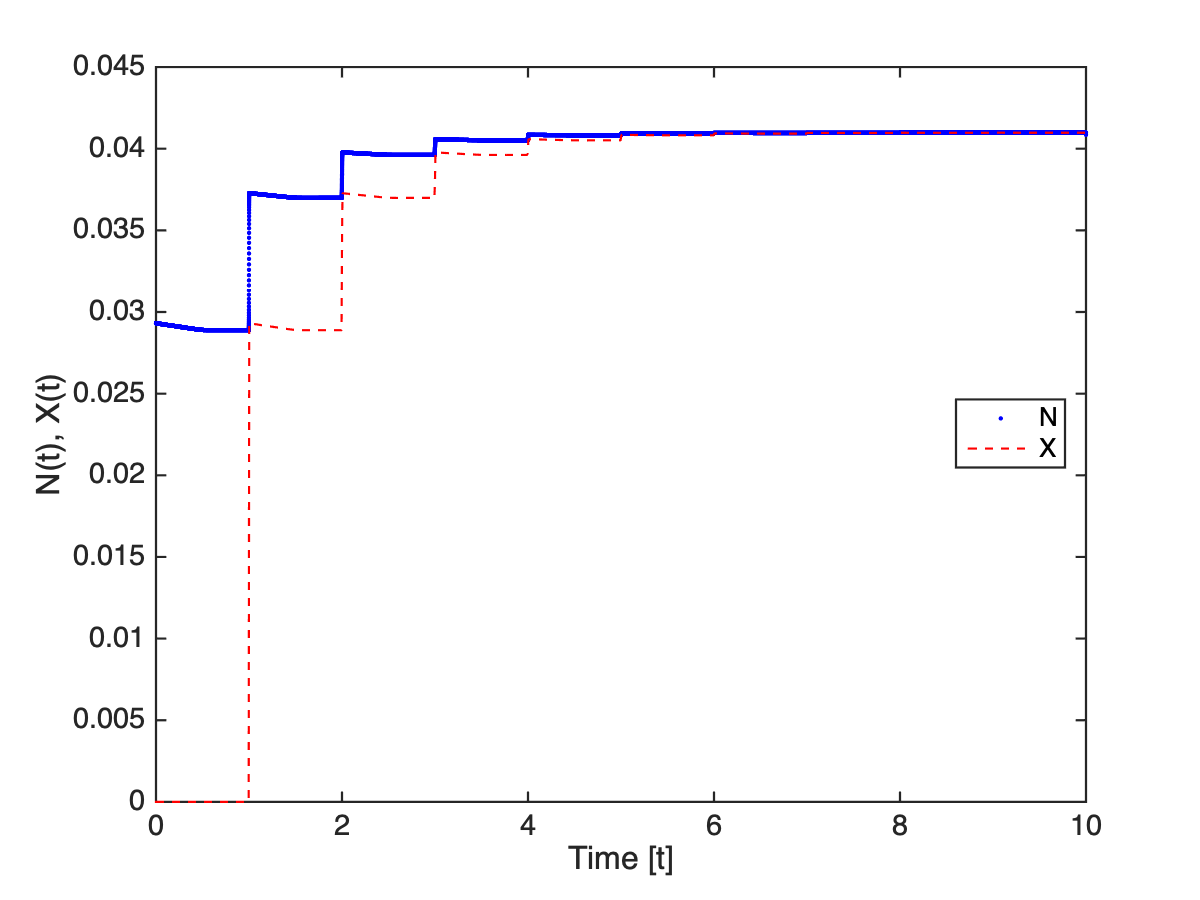} \\   \end{tabular}
    \caption{An excitatory case for the \eqref{eq:delay} equation. (a-b) Density $n(t,s)$, discharging flux $N(t)$ and total activity $X(t)$ for the \eqref{eq:delay} equation with $\alpha(t)\approx \delta(t)$. (c-d) Same variables of the system with $\alpha(t)\approx\delta(t-1)$.}
    \label{fig:exp_excita_r3_delay}
\end{figure}


\subsection*{Example 4: A variable refractory period \cite{pakdaman2009dynamics}}
Based on Example 2 in \cite{pakdaman2009dynamics}, we consider a firing coefficient with variable refractory period as in Equation \eqref{p-esp2} for the \eqref{eq:delay} equation with parameters given by 
\begin{equation}\label{exp4:PPS}
  p(s,X)=\chi_{\{s>\sigma(X)\}}(s),\quad  \sigma(X)=2-\frac{X^4}{X^4+1},  \quad \alpha(t)=J\alpha_1(t),\quad n^0(s)=e^{-(s-1)}\chi_{\{s>1\}}(s),
\end{equation}
where $J>0$ is the connectivity parameter of the network. As it was studied in \cite{pakdaman2009dynamics}, the system has different behaviors depending value of $J$. When $J$ is small the network is weakly connected and the dynamics are close to the linear case, while if $J$ is large the network is strongly connected and different asymptotic behaviors are possible. We recall that when $\alpha(t)=\delta_0$, we formally obtain the \eqref{eq:model} equation where 
\begin{equation}\label{exp4:PPSlocal}
 X(t)=JN(t),\quad\text{and},\quad  p(s,N)=\chi_{\{s>\sigma(JN)\}}(s).
\end{equation}

Taking $J=2.5$ as in \cite{pakdaman2009dynamics}, in Figure~\ref{fig:PPS} we compare the numerical approximation of $N(t)$ with $t\in[0,14]$, for both equations \eqref{eq:model} and \eqref{eq:delay}. In Figure \ref{fig:PPS}(a) we observe that the solution of the \eqref{eq:model} equation is asymptotic to a periodic pattern with jump discontinuities, where this type of solutions were also observed in \cite{torres2021elapsed}. We also observe that when a jump discontinuity arises for $N(t)$ in Figure \ref{fig:PPS}(a), the function $\Psi(N,n)$ is close to zero as we see in Figure \ref{fig:PPS}(b), verifying numerically the invertibility condition \eqref{cond:inverti} that ensures the continuity of solutions.

In Figure \ref{fig:PPS}(c) we observe that the discharging flux $N(t)$ in the \eqref{eq:delay} equation is a smooth approximation of the solution observed in Figure \ref{fig:PPS}(a) and we see the same phenomenon for $\Tilde{X}(t)\coloneqq X(t)/J$, corresponding to a normalization of the total activity in order to compare these quantities. Finally in Figure~\ref{fig:PPS}(d) we display the corresponding numerical approximation of the \eqref{eq:delay} equation with $n(t,s)$ for $(t,s)\in[0,14]\times[0,8]$, which also follows a periodic pattern.

\begin{figure}[t]
 \begin{tabular}{cc}
     (a) & (b)  \\
\includegraphics[width=0.45\textwidth]{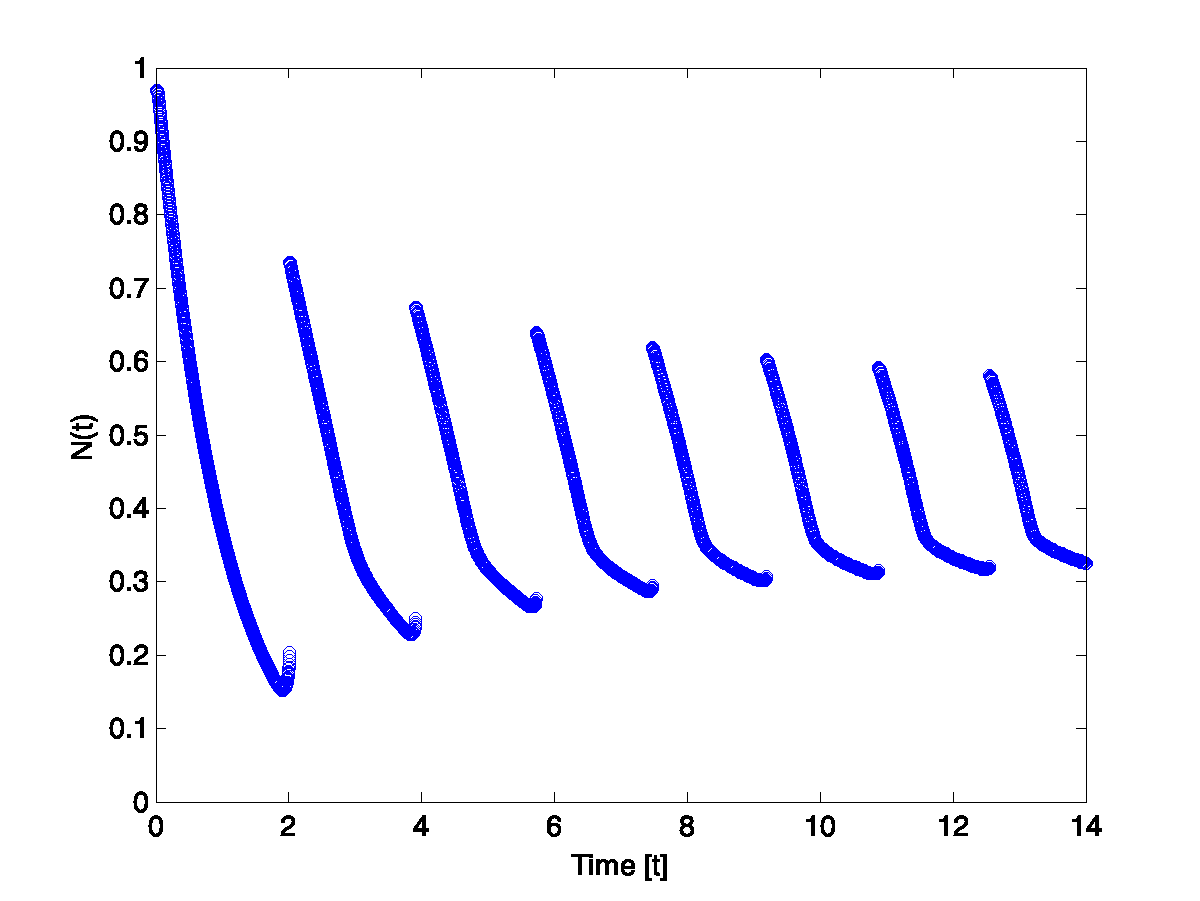}  &\includegraphics[width=0.45\textwidth]{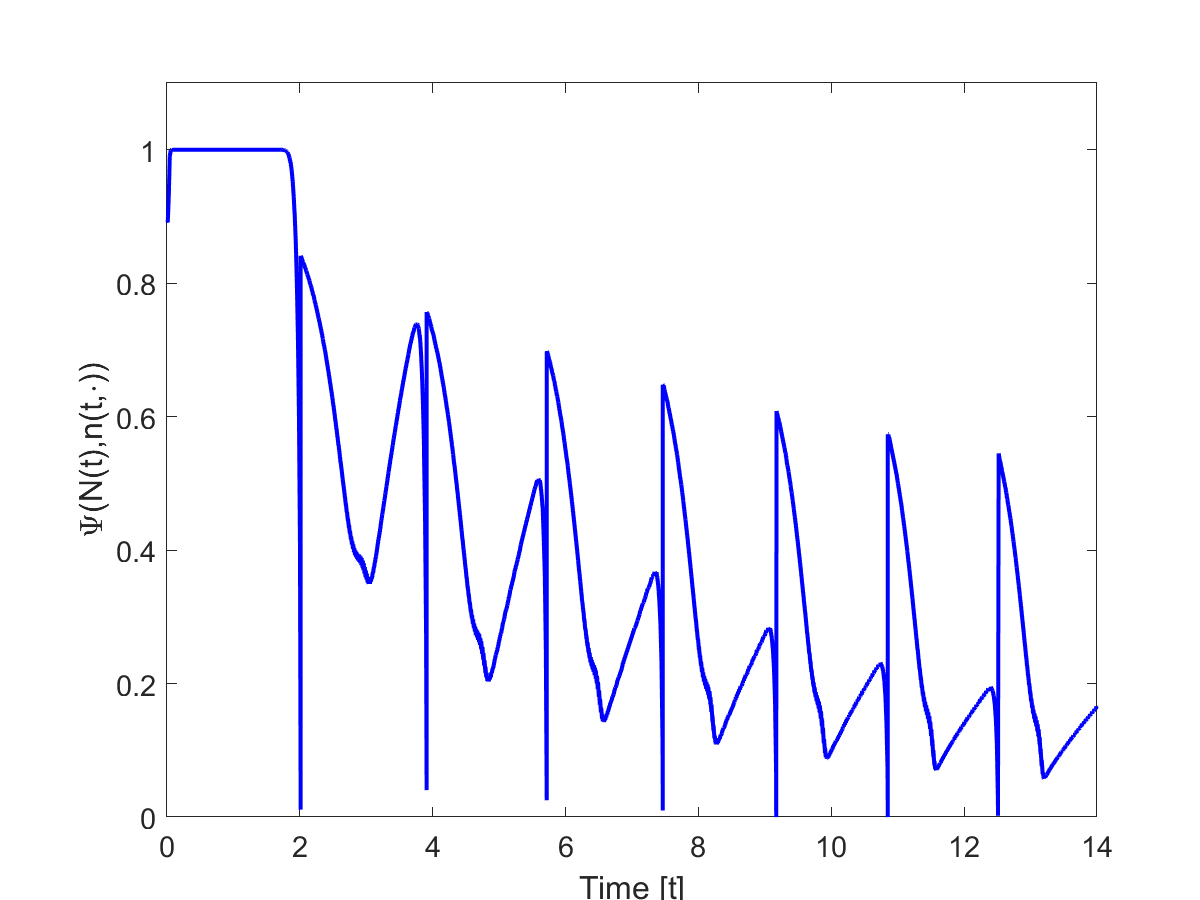} \\
     (c) & (d)  \\
\includegraphics[width=0.45\textwidth]{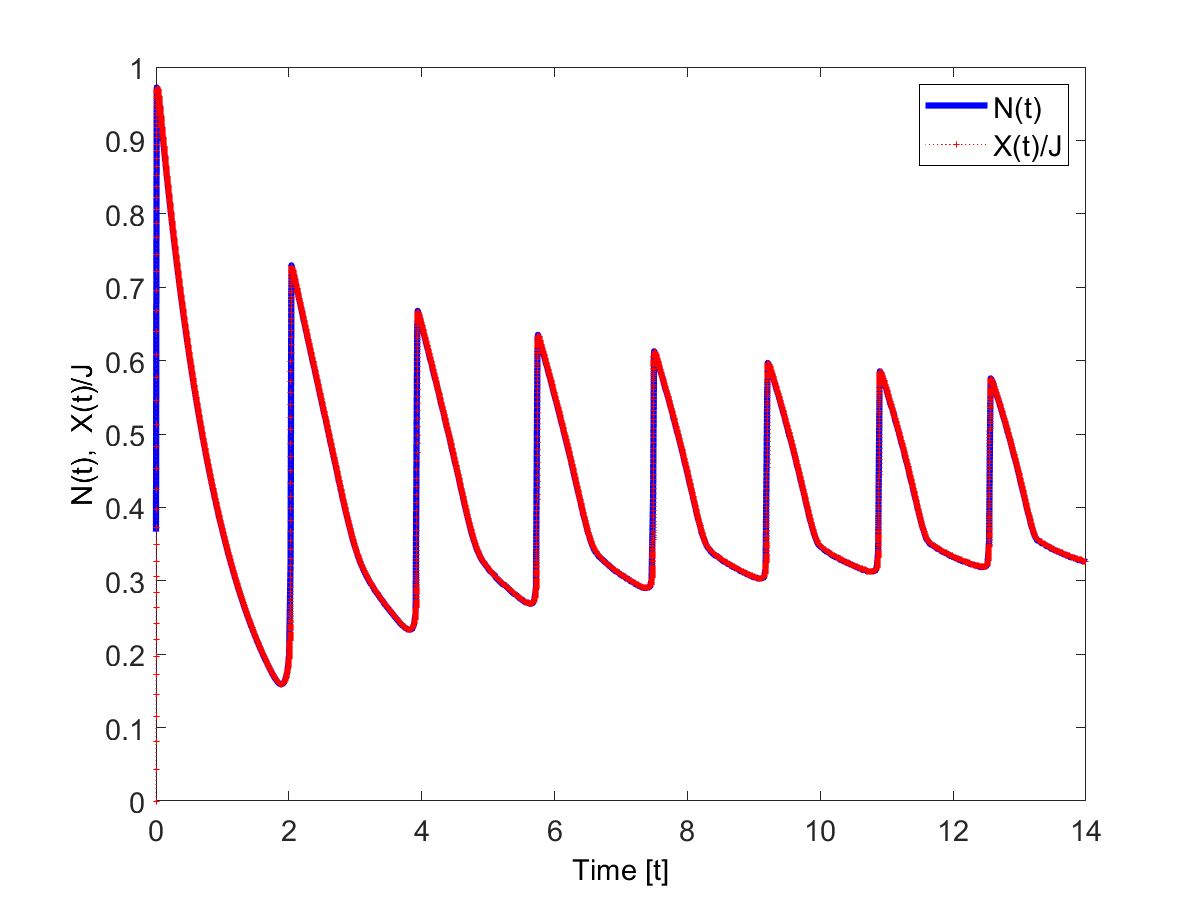} &
\includegraphics[width=0.45\textwidth]{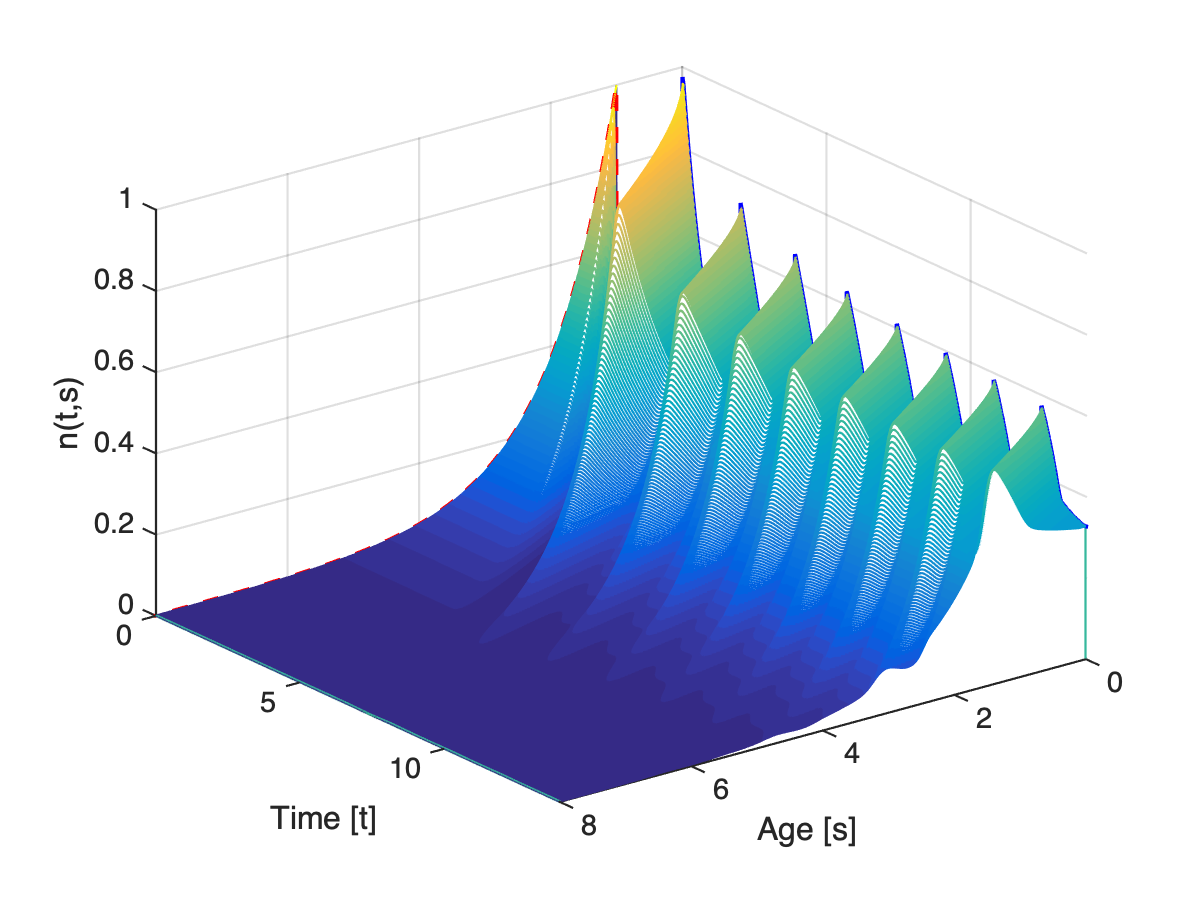}\\
 \end{tabular}
    \caption{firing coefficient with variable refractory period. Comparing numerical approximation of $N(t)$ for both equations \eqref{eq:model} and \eqref{eq:delay} . (a) Discharging flux $N(t)$ for \eqref{eq:model} equation, (b) Invertibility condition $\Psi(N,n)$ for \eqref{eq:model}. (c) $N(t)$ and $X(t)/J$ for \eqref{eq:delay} equation, (d) Density $n(t,s)$ for \eqref{eq:delay}.}
    \label{fig:PPS}
\end{figure}

\section*{Conclusion and perspectives}
In this article we managed to improve the proof of \cite{pakdaman2009dynamics,canizo2019asymptotic} on well-posedness for the instantaneous transmission equation by stating a general invertibility criterion, allowing to extend the theory for wider types of firing coefficients $p$, including for example the strongly inhibitory case. This criterion establishes the necessary condition to extend a solution in a continuous way, even in the case we do not have uniqueness, and we get the necessary condition for a solution to have jump discontinuities. The key idea is to apply the implicit function theorem to the correct fixed-point problem and the arguments can be extended when this rate is not necessarily bounded. This motivates the study of elapsed time model when the activity of neurons may increase to infinity and blow-up or other special phenomena might arise.

Another interesting question is convergence of the delay kernel $\alpha(t)$ to the Dirac's mass $\delta(t)$ in order to compare both equations \eqref{eq:model} and \eqref{eq:delay}. We conjecture that the total activity $X(t)$ of the \eqref{eq:delay} equation converges almost everywhere (or in some norm) to the discharging flux $N(t)$ of the \eqref{eq:model} equation. In particular we believe that the convergence holds for every $t>0$ except when $N(t)$ has a jump discontinuity. This motivates to determine if the assumption of instantaneous transmission is actually a good approximation of the neural dynamics, which have indeed a certain delay. As in Remark \ref{past}, the definition of $N(t)$ when $t<0$ in for the \eqref{eq:delay} will determine which solution of the \eqref{eq:model} is recovered when $\alpha$ converges to a Dirac's mass, in case of multiple solutions of \eqref{eq:model}.

Similarly, when the delay kernel $\alpha(t)$ converges to the Dirac's mass $\delta(t-d)$ in the sense of distributions, we conjecture that solutions of the \eqref{eq:delay} equations converge to the solutions of Equation \eqref{eq:delaypuro}. 

From a numerical point of view, we proved the convergence of the adapted explicit upwind scheme for the elapse time model relying on the mass-preserving property, the analysis of the fixed-point equations \eqref{eqfijo-N}, \eqref{eqfijo-X} and the key BV-estimate, from where we obtain the compactness to conclude the result. Despite being a simple and robust scheme, the main limitations of our approach is that we consider an explicit scheme that forces CFL condition and we need the initial data to have enough regularity, while other methods are potentially able to deal with measure or less regular $L^1$ functions.

This study of the elapsed-time model can be also extended by considering implicit or semi-discrete schemes, but a more detailed analysis on the mass conservation and entropy estimates must be considered. Other possible discretizations to solve the equations include for the example high-order Runge-Kutta-WENO methods (see for example \cite{leveque1992numerical,godlewski1996numerical,cockburn1998essentially,bouchut2004nonlinear}). These alternatives might be useful to deal with potential jump discontinuities and when the coefficient $p$ is not bounded, which implies that the total activity may be also unbounded. Furthermore, this numerical analysis can be considered for other extensions of the elapsed time equation such as the model with fragmentation \cite{pakdaman2014adaptation}, spatial dependence \cite{torres2020dynamics,dumont2024oscillations}, the multiple-renewal equation \cite{torres2022multiple} and the model with leaky memory variable in \cite{fonte2022long} or other type of structured equations.

\subsubsection*{Acknowledgements}
This work has been supported by ANID project ECOS200018. MS and LMV was partially supported by ANID-Chile through Centro de Modelamiento  Matem\'atico  (FB210005) and INRIA Associated team ANACONDA. 
MS was supported by  Fondecyt-ANID project 1220869, and Jean d'Alembert fellowship program, Université de Paris-Saclay.
NT was supported by the grant Juan de la Cierva FJC2021-046894-I funded by MCIN/AEI with the European Union NextGenerationEU/PRTR, the project EUR SPECTRUM with the initiative IDEX of Université de Côte d'Azur and by the Center for Mathematical Modeling (CMM) BASAL fund FB210005 for center of excellence from ANID-Chile.

\subsubsection*{ORCID iDs}
Mauricio Sepulveda: \href{https://orcid.org/0000-0001-8463-3830}{https://orcid.org/0000-0001-8463-3830}.\\
Nicolas Torres: \href{https://orcid.org/0000-0001-6059-9754}{https://orcid.org/0000-0001-6059-9754}.\\
Luis Miguel Villada: \href{https://orcid.org/0000-0002-4860-4431}{https://orcid.org/0000-0002-4860-4431}.

\typeout{}
\bibliography{main.bib}
\bibliographystyle{vancouver}

\end{document}